\documentclass[11pt,a4paper]{article}

\usepackage{amssymb}
\usepackage[usenames]{color}
\usepackage{theorem}
\usepackage[utf8]{inputenc}
\usepackage{enumerate}
\usepackage{amsmath}
\usepackage{comment}

\numberwithin{equation}{section}

\newtheorem{theorem}{Theorem}[section]
\newtheorem{proposition}[theorem]{Proposition}
\newtheorem{lemma}[theorem]{Lemma}
\newtheorem{corollary}[theorem]{Corollary}
\newtheorem{definition}[theorem]{Definition}
\newtheorem{remark}[theorem]{Remark}

\newenvironment{proof}{\noindent\textbf{Proof}}
                      {$\Box$\vskip\theorempostskipamount}

\begin{document}

\title{\textbf{Twisted associativity of the cyclically reduced product of words, part 1}}

\author{Carmelo Vaccaro}

\date{}

\maketitle

\begin{abstract} The cyclically reduced product of two words $u, v$, denoted $u * v$, is the cyclically reduced form of the concatenation of $u$ by $v$. This product is not associative. Recently S. V. Ivanov has proved that the Andrews-Curtis conjecture can be restated in terms of the cyclically reduced product and cyclic permutations instead of the reduced product and conjugations.

In a previous paper we have started a thorough study of $*$ and of the structure of the set of cyclically reduced words $\hat{\mathcal{F}}(X)$ equipped with $*$. In particular we have found that a certain number of properties of the free group equipped with the reduced product can be generalized to $(\hat{\mathcal{F}}(X), *)$.

In this paper we continue this study by proving that a generalized version of the associative property holds for $*$ in a special case. In a following paper we will prove that a more general version of the associative property holds for any case.
 \end{abstract}

\smallskip \smallskip


\smallskip \smallskip


\section*{Introduction}

Let $X$ be a set of letters, let $X^{-1}$ be the set of inverses of elements of $X$ and let $\mathcal{M}(X \cup X^{-1})$ be the free monoid on $X \cup X^{-1}$. The elements of $\mathcal{M}(X \cup X^{-1})$ are the non-necessarily reduced words on $X$. We denote $\mathcal{F}(X)$ the free group on $X$ and we consider it as the subset of $\mathcal{M}(X \cup X^{-1})$ consisting of reduced words. We denote $\hat{\mathcal{F}}(X)$ the set of cyclically reduced words on $X$.

Given $v, w \in \mathcal{M}(X \cup X^{-1})$ the cyclically reduced product of $v$ by $w$, denoted $v*w$, is defined as the cyclically reduced form of the concatenation $vw$. By contrast, the reduced product of $v$ by $w$, which we denote $v \cdot w$, is defined as the reduced form of $vw$.

The cyclically reduced product has applications to the Andrews-Curtis conjecture: in \cite{Ivanov06} and \cite{Ivanov18} S. V. Ivanov has proved that the conjecture (with and without stabilizations) is true if and only if in the definition of the conjecture we replace the operations of reduced product and conjugations with the cyclically reduced product and cyclic permutations. The importance of this result stems from the fact that while there are infinitely many conjugates of one word, there are only finitely many cyclic permutations, thus making much easier the search of Andrews-Curtis trivializations by enumerations of relators, like for example the approaches used in \cite{BMc} or \cite{PU}. 

\smallskip

The set $\hat{\mathcal{F}}(X)$ is closed with respect to $*$, like the free group is closed with respect to $\cdot$, but the structure of $(\hat{\mathcal{F}}(X), *)$ is much less nice than that of $(\mathcal{F}(X), \cdot)$, mainly because $*$ is not associative.

However $(\hat{\mathcal{F}}(X), *)$ has interesting properties. Indeed we have proved in \cite{FirstArticle} the following facts concerning $\hat{\mathcal{F}}(X)$. There is a unique identity element and each element has a unique inverse. There are no idempotents except 1. The Latin square property does not hold because $*$ is not cancellative, however a generalization of this property holds true: given $u, w \in \hat{\mathcal{F}}(X)$ with $w \neq 1$ there exist infinitely many pairs of words $v_1, v_2 \in \hat{\mathcal{F}}(X)$ such that $u*v_1 = v_2*u = w$; moreover $v_1$ and $v_2$ are cyclic permutations one of the other. This is analogous to the fact that for $u, w \in \mathcal{F}(X)$ there exists a (unique) pair of words $v'_1, v'_2 \in \mathcal{F}(X)$ such that $u \cdot v'_1 = v'_2 \cdot u = w$; moreover $v'_1$ and $v'_2$ are conjugate.

\smallskip

As the latter example shows, it appears that the structure of $\hat{\mathcal{F}}(X)$ equipped with the cyclically reduced product and with cyclic permutations enjoys similar properties as those enjoyed by the free group equipped with the reduced product and conjugations. This was hinted at in the papers \cite{Rourke}, \cite{Scarabotti}, \cite{Ivanov06}, \cite{Ivanov18}. In \cite{FirstArticle} we have explored further this fact; in particular we have proved that for words $u$ and $v$ the cyclically reduced product $u*v$ is a cyclic permutation of $v*u$ and the identity among relations that follows from this fact is a generalization of the identity among relations that follows from the fact that in the free group $u \cdot v$ is a conjugate of $v \cdot u$.

The main result of this paper goes further in this direction. Indeed, as seen above, in the free group we have the following fact: let $u, w$ be words and let us set $v'_1 := u^{-1} \cdot w$ and $v'_2 := w \cdot u^{-1}$; then $u \cdot v'_1 = v'_2 \cdot u = w$. 

If $X$ has at least two elements this result is not true in $\hat{\mathcal{F}}(X)$, that is if we set $v_1 := u^{-1} * w$ and $v_2 := w * u^{-1}$ then it is not true in general that $u * v_1 = w$ and $v_2 * u = w$. Indeed let $x, y \in X$ be such that $x \neq y$ and let $u := xy$ and $w := y^2$; then $u * v_1 = x y x^{-1} y \neq w$, while $v_2 * u = w$. Let $u := xy$ and $w := x^2$; then $u * v_1 = w$, while $v_2 * u = x y^{-1} x y \neq w$. Finally if $u := y x y$ and $w := y x^{-1} y$, then $u * v_1 = y x y x^{-2} \neq w$, $v_2 * u = x^{-2} y x y \neq w$ and $u * v_1 \neq v_2 * u$.

However we prove in Theorem \ref{mainTheor1} that a generalization of this result that makes use of cyclic permutations holds true: there exist cyclic permutations $u'$ and $u''$ of $u$ such that either $u' * v_1$ and $v_2 * u''$ are cyclic permutations of $w$ or there exists a non-empty word $h$ such that the concatenations $u' h v_1 h^{-1}$ and $h v_2 h^{-1} u'$ are cyclically reduced and are cyclic permutations of $w$.

The analogy with the above result in the free group does not stop here. In the free group we have the following fact. Let $u$ and $w$ be relators of a group presentation; then the equalities $u \cdot v'_1 = w$ and $v'_2 \cdot u = w$ determine the following identities among relations: $u \centerdot u^{-1} \centerdot w \equiv w$ and $w \centerdot u^{-1} \centerdot u \equiv w$. These identities are the simplest possible types of identities among relations and we have called them \textit{strictly basic} (Definitions \ref{defBasIAR} and \ref{defBasIAR2}). We prove that the same is true for the cyclically reduced product: the identities among relations following from the fact that $w$ is a cyclic permutation of $u' * v_1$ and $v_2 * u''$ or of $u' h v_1 h^{-1}$ and $h v_2 h^{-1} u'$ are strictly basic. 

\smallskip

The result of Theorem \ref{mainTheor1} concerns a special case of the associative property. If $*$ were associative then for any $u, v, w$ we would have that $u * (v * w) = (u * v) * w$. In particular by taking $v = u^{-1} * w$ this would imply that $u * v_1 = w$. Indeed we would have that $u * v_1 = u * (u^{-1} * w) = (u * u^{-1}) * w = w$, because $u * u^{-1} = 1$ and $1 * w = w$. Therefore Theorem \ref{mainTheor1} says that a generalization of the associative property (a ``twisted" version of it) holds in the special case where $v = u^{-1}$.

In \cite{ThirdArticle} we will prove that a more general version of the associative property than that of Theorem \ref{mainTheor1} holds true for any $v$.

\bigskip

\noindent \textbf{Structure of the paper.}

In Section \ref{sec1} we give the basic definitions and prove some elementary results about the reduced and the cyclically reduced product, cyclic permutations and reversions of words. In Section \ref{sec2} we prove the main result of the paper. In Appendix \ref{secA} we define and prove some facts concerning identities among relations used in the main theorem. Finally in Appendix \ref{secB} we prove some technical results needed for the proof of the main result of the paper.

This paper is a logical continuation of \cite{FirstArticle} but can be read independently of the latter and is self contained. All the results from \cite{FirstArticle} needed in this paper are stated in Section \ref{sec1}.

\section{Words, cyclic permutations and cyclically reduced product} \label{sec1}

Let $Y$ be a set and let us consider $\mathcal{M}(Y)$, the free monoid on $Y$. The elements of $Y$ are called \textit{letters}, those of $\mathcal{M}(Y)$ \textit{the words in $Y$}. As usual given words $v, w \in \mathcal{M}(Y)$ we will denote $vw$ the product of $v$ by $w$, which is the concatenation of the words $v$ and $w$. The word with no letters, which is the identity element of $\mathcal{M}(Y)$, is denoted 1.

Let $v, w \in \mathcal{M}(Y)$; we say that $w$ is a \textit{subword of $v$} if there exist $p, q \in \mathcal{M}(Y)$ such that $v = p w q$. In this case we say that $w$ is a \textit{prefix of $v$} if $p = 1$ and that $w$ is a \textit{suffix of $v$} if $q = 1$.

Let $v = y_1 \dots y_n \in \mathcal{M}(Y)$ with $y_1, \dots, y_n \in Y$. The \textit{length of $v$} is defined as $|v| := n$. The \textit{reverse of $v$} is defined as the word $\underline{v} := y_n \dots y_1$, where the order of the letters is the reverse as that of $v$. 

Let $w_1$ and $w_2$ be words and let $w:=w_1 w_2$. The word $w_2 w_1$ is called a \textit{cyclic permutation of $w$}. Given two words $u$ and $v$ the relationship ``$u$ is a cyclic permutation of $v$" is an equivalence that we denote $u \sim v$.

Let $X$ be a set; we denote $\mathcal{F}(X)$ the free group on $X$ and we consider $\mathcal{F}(X)$ as a subset\footnote{Usually $\mathcal{F}(X)$ is considered a quotient of $\mathcal{M}(X \cup X^{-1})$, but in this paper we will not follow this habit.} of $\mathcal{M}(X \cup X^{-1})$. In particular $\mathcal{F}(X)$ is the set of \textit{reduced words on $X$}, i.e., the words of the form $x_1 \dots x_n$ with $x_i \in X \cup X^{-1}$ and $x_{i+1} \neq x_i^{-1}$ for $i = 1, \dots, n-1$.

We denote $\rho: \mathcal{M}(X \cup X^{-1}) \rightarrow \mathcal{F}(X)$ the function sending a word to its unique \textit{reduced form}. Given $u, v \in \mathcal{F}(X)$ the product\footnote{The product of two reduced words in $\mathcal{F}(X)$ does not coincide with the product of the same words in $\mathcal{M}(X \cup X^{-1})$. In particular $\mathcal{F}(X)$ is not a subgroup of $\mathcal{M}(X \cup X^{-1})$.} of $u$ by $v$ in $\mathcal{F}(X)$ is $\rho(uv)$, the reduced form of $uv$. This product will be denoted $u \cdot v$.

\medskip

\noindent \textbf{Convention} In this paper we adopt the following conventions:

1. With the term \textit{word} we mean a non-necessarily reduced word, i.e., an element of $\mathcal{M}(X \cup X^{-1})$.

2. Given $u$, $v_1, \dotsc, v_n \in \mathcal{M}(X \cup X^{-1})$, with the notation $u = v_1 \dots v_n$ we mean the equality in $\mathcal{M}(X \cup X^{-1})$ of $u$ with the concatenation of words $v_1 \dots v_n$ even if $u$ and all the $v_j$ belong to $\mathcal{F}(X)$. This kind of equality is called a \textit{factorization} of $v$ in the Combinatorics of Words literature (see \cite{Kar}, pag. 2 or \cite{ChofKar}, pag. 332). The equality in $\mathcal{F}(X)$ of $u$ with the reduced product of $v_1, \ldots, v_n$ will be denoted by $u = v_1 \cdot \, \ldots \, \cdot v_n$ and corresponds to the equality $\rho(u) = \rho(v_1 \ldots v_n)$ in $\mathcal{M}(X \cup X^{-1})$. 

\medskip

The operations of inversion and reversion of words commute one with the other, that is given a word $v$ we have that $(\underline{v})^{-1} = \underline{(v^{-1})}$. Thus we will denote $\underline{v}^{-1}$ the inverse of the reversion of $v$ without fear of ambiguity.

We say that a reduced word is \textit{cyclically reduced} if its last letter is not the inverse of the first one, that is if all its cyclic permutations are reduced. We denote $\hat{\mathcal{F}}(X)$ the set of cyclically reduced words.

Given a word $w$ there exist unique $t \in \mathcal{F}(X)$ and $c \in \hat{\mathcal{F}}(X)$ such that $\rho(w) = t c t^{-1}$. The word $c$ is called the \textit{cyclically reduced form of $w$} and is denoted $\hat{\rho}(w)$. In particular we consider the function $\hat{\rho}: \mathcal{M}(X \cup X^{-1}) \rightarrow \hat{\mathcal{F}}(X)$ sending a word to its unique cyclically reduced form. Therefore we have that $\rho(w) = t \hat{\rho}(w) t^{-1}$ and that $\rho(w)$ is cyclically reduced if and only if $t=1$. 

Given words $u$ and $v$ we denote $u*v$ the \textit{cyclically reduced product of $u$ by $v$}, i.e., $u * v := \hat{\rho}(uv)$. This product is non-associative. Indeed let $u = xy$, $v = x^{-1}$ and $w = x$; then $(u*v)*w = yx$ while $u*(v*w) = xy$.

\begin{proposition} \label{summar} \rm Let $u, v, u_1, u_2, \dots, u_n$ be words; then the following results hold: 	
	\begin{enumerate} [(1)]
		\item \label{reverse2} The reverse of $u_1 u_2 \dots u_n$ is the word $\underline{u_n} \dots \underline{u_2} \, \underline{u_1}$. 
		
		\item \label{remCycPerm} $u$ is a cyclic permutation of $v$ if and only if there exists a word $p$ such that $up = pv$. 
		
		\item \label{reversesim} If $u \sim v$ then $\underline{u} \sim \underline{v}$. 
				
		\item \label{cpCanc} Let $u$ be a cyclic permutation of $\rho(v)$. Then there exists a cyclic permutation $v'$ of $v$ such that $\rho(v') = \rho(u)$. 

		\item \label{u*vNonRed} $u*v = \rho(u) * \rho(v)$. More generally if $u_1, v_1$ are words such that $\rho(u_1) = \rho(u)$ and $\rho(v_1) = \rho(v)$, then $u*v = u_1 * v_1$. 

		\item \label{scope} The cyclically reduced form of $u$ is equal to the reduced form of some conjugate of $u$, that is there exists a word $\alpha$ such that $\hat{\rho}(u) = \rho(\alpha u \alpha^{-1})$.  

		\item \label{permCon2} Let $u$ be a cyclic permutation of $v$; then the reduced form of $u$ is the reduced form of some conjugate of $v$. 
		
		\item \label{revCRP} The reverse of $u*v$ is equal to $\underline{v}*\underline{u}$ and the cancellations made to obtain $\underline{v}*\underline{u}$ from $\underline{v} \, \underline{u}$ are the reverse of those made to obtain $u*v$ from $uv$. 
		
		\item \label{revCycRedForm} $\hat{\rho}(\underline{u}) = \underline{\hat{\rho}(u)}$ and the cancellations made to obtain $\hat{\rho}(\underline{u})$ from $\underline{u}$ are the reverse of those made to obtain $\hat{\rho}(u)$ from $u$. 
	\end{enumerate} 
\end{proposition}

\begin{proof} 
	\begin{enumerate} [(1)]

		\item See Remark 1.1 of \cite{FirstArticle}.
		
		\item See Prop. 1.3.4 of \cite{Loth} or Theor. 4 of \cite{Kar}).
		
		\item See Remark 1.5 of \cite{FirstArticle}.
		
		\item See Remark 1.12 of \cite{FirstArticle}.

		\item See Remark 2.14 of \cite{FirstArticle}. 

		\item See Remark 2.4 of \cite{FirstArticle}. 
		
		\item See Remarks 1.14 of \cite{FirstArticle}.
	
		\item See Remark 2.15 of \cite{FirstArticle}.	
		
		\item See Proposition 2.11 of \cite{FirstArticle}.			
	\end{enumerate} 
\end{proof}

\begin{remark} \label{uvw} \rm  The following result is obvious: if $u, v, w$ are words such that
$uv$ and $vw$ are reduced and $v \neq 1$ then $uvw$ is reduced. \end{remark}

\begin{remark} \label{commFG} \rm Two elements of a free group commute if and only if they are power of the same element, i.e., if $u, v \in \mathcal{F}(X)$ are such that $\rho(uv) = \rho(vu)$ then there exist $c \in \mathcal{F}(X)$ and $m, n \in \mathbb{Z}$ such that $u = \rho(c^m)$ and $v = \rho(c^n)$ (see Proposition I.2.17 of \cite{LS}).

In particular let $a, b, u \in \mathcal{F}(X)$ be such that $\rho(a u a^{-1}) = \rho(b u b^{-1})$. This implies that $\rho(b^{-1} a u) = \rho(u b^{-1} a)$ and since $\rho(b^{-1} a)$ and $u$ commute, there exist $c \in \mathcal{F}(X)$ and $m, n \in \mathbb{Z}$ such that $u = \rho(c^m)$ and $\rho(b^{-1} a) = \rho(c^n)$, in particular $a = \rho(b c^n)$.
\end{remark}

\begin{remark} \label{LeviLemma} \rm We have the following result, known as \textit{Levi's Lemma} (see \cite{ChofKar}, pag. 333 or \cite{Kar}, Theor. 2): \textit{let $u_1, u_2, v_1, v_2$ be words such that $u_1 u_2 = v_1 v_2$; then there exists a word $p$ such that either $u_1 = v_1 p$ and $v_2 = p u_2$ or $v_1 = u_1 p$ and $u_2 = p v_2$.} The two cases can be represented graphically in the following way,
	
\medskip
	
\hspace{2cm} 
\begin{tabular}{|c|c|}
	\hline 
	$\,\, u_1 \,$   & $u_2 \,$ \\
	\hline
\end{tabular}
	\quad  \hspace{1cm} and \hspace{1cm}
\begin{tabular}{|c|c|}
	\hline 
	$u_1$ & $\,\,u_2\,$ \\
	\hline
\end{tabular}
	
\hspace{2cm}
\begin{tabular}{|c|c|}
	\hline 
	$v_1$ & $\,\, v_2 \,\,\,$ \\
	\hline
\end{tabular}
	\quad \hspace{28.8mm}
\begin{tabular}{|c|c|}
	\hline 
	$\,\,v_1 \,\,$   & $v_2$ \\
	\hline
\end{tabular}
	
\medskip
	
\noindent and correspond to putting the bar separating $v_1$ and $v_2$ either inside $u_1$ or inside $u_2$. The case when this bar is exactly below that separating $u_1$ and $u_2$, i.e., when $u_1 = v_1$ and $u_2 = v_2$, can be considered a special case of both the cases. 
	
In general let us consider the word equation $u_1 \dots u_m = v_1 \dots v_n$, possibly with $m \neq n$. Any solution to this equation determines uniquely a way of putting $n - 1$ bars inside the $m$ spaces corresponding to $u_1, \dots, u_m$ and also a way of putting $m - 1$ bars inside the $n$ spaces corresponding to $v_1, \dots, v_n$. This is true even if some of the $u_i$ or $v_j$ are the empty word. Indeed if $u_i = 1$ or $v_j = 1$ then no bar must be contained in $u_i$ or $v_j$. 
	
We observe that a solution to the equation $u_1 \dots u_m = v_1 \dots v_n$ determines also a \textit{weak composition}\footnote{a weak composition for an integer is a composition when 0's are allowed} for $n - 1$ in $m$ parts and for $m - 1$ in $n$ parts.
	
We give the following as an example for $m = 4$ and $n = 3$:
	
\medskip

\hspace{3.5 cm}
\begin{tabular}{|c|c|c|c|}
	\hline 
	$u_1$ & $u_2$ &	$u_3$ & $\,\, u_4 \,\,$ \\
	\hline
\end{tabular}
	
\hspace{3.5 cm}
\begin{tabular}{|c|c|c|}
	\hline 
	$\,\, v_1 \,\,$ & $\,\,\,\,\,\,\, v_2 \,\,\,\,\,\,\,$ & $v_3$\\
	\hline
\end{tabular}

\medskip
	
\noindent Here we can say that there exist words $a, b, c$ such that $v_1 = u_1 a$, $u_2 = a b$, $v_2 = b u_3 c$ and $u_4 = c v_3$. This solution determines the weak compositions $(0, 1, 0, 1)$ for 2 and $(1, 2, 0)$ for 3, which are obtained by counting the number of bars inside each $u_i$ and each $v_j$ respectively.
\end{remark}

\begin{lemma} \label{shirv} Let $u$ and $v$ be reduced words such that $u \neq v^{-1}$. Then one of the following holds:
	\begin{enumerate} [1)]
		\item there exist words $u_1, a, s$ such that $u = u_1 a$, $v = a^{-1} s (u*v) s^{-1} u_1^{-1}$ and $\rho(uv) = u_1 s (u*v) s^{-1} u_1^{-1}$;
		
		\item there exist non-empty words $c_1, c_2$ and words $t, a$ such that $u*v = c_1 c_2$, $u = t c_1 a$, $v = a^{-1} c_2 t^{-1}$, $\rho(uv) = t c_1 c_2 t^{-1}$, $\rho(v u) = a^{-1} c_2 c_1 a$ and $v*u = c_2 c_1$; 
		
		\item there exist words $v_1, s, a$ such that $u = v_1^{-1} s (u*v) s^{-1} a$, $v = a^{-1} v_1$ and $\rho(uv) = v_1^{-1} s (u*v) s^{-1} v_1$. 
	\end{enumerate}
\end{lemma} 

\begin{proof} See Lemma B.2 of \cite{FirstArticle}.
\end{proof}

\begin{proposition} \label{puzo} Let $u$ and $v$ be words; then $u*v$ is a cyclic permutation of $v*u$. 
	
Moreover if $u$ and $v$ are reduced and if there exist words $\alpha$, $\beta$, $u'$, $v'$ such that $u = \alpha u' \beta$ and $v = \beta^{-1} v' \alpha^{-1}$ then the words $\beta \beta^{-1}$ and $\alpha^{-1} \alpha$ are canceled when obtaining $u*v$ from $uv$ and when obtaining $v*u$ from $vu$.
\end{proposition} 

\begin{proof} See Theorem 4.1 of \cite{FirstArticle}.
\end{proof}

\begin{corollary} \label{cycRedPerm} Let $w$ be a word and let $w'$ be a cyclic permutation of $w$. Then $\hat{\rho}(w')$ is a cyclic permutation of $\hat{\rho}(w)$. \end{corollary}

\begin{proof} See Corollary 4.3 of \cite{FirstArticle}. \end{proof}

\begin{corollary} \label{permCycRedForm} If $t, w$ are words then $\hat{\rho}(t w t^{-1})$ is a cyclic permutation of $\hat{\rho}(w)$. If moreover $\rho(t)\rho(w)\rho(t)^{-1}$ is reduced then $\hat{\rho}(t w t^{-1}) =\hat{\rho}(w)$. \end{corollary} 

\begin{proof} See Corollary 4.4 of \cite{FirstArticle}. \end{proof}

\section{The main result}  \label{sec2}

This section uses the results of Appendix \ref{secB}.


Since $\mathcal{F}(X)$ is a group we know that given $u, w \in \mathcal{F}(X)$ we have that 
	\begin{equation} \label{assocFG1} u \cdot (u^{-1} \cdot w) = (u \cdot u^{-1}) \cdot w = w\end{equation}
and
	\begin{equation} \label{assocFG2} (w \cdot u^{-1}) \cdot u = w \cdot (u^{-1} \cdot u) = w\end{equation}
This is not true in general for the cyclically reduced product $*$ if $X$ has at least two elements. Indeed let $x, y \in X$ with $x \neq y$ and let $u := x y^{-1} x y^2$, $w := x^2 y^{-1} x^{-1} y^2$. Then 
		\begin{equation} \label{assocCFG1} u * (u^{-1} * w) = y^{-1} x y^2 x^{-1} y x y^{-1} \neq w\end{equation}
and
		\begin{equation} \label{assocCFG2} (w * u^{-1}) * u = x y^{-1} x^{-2} y x y^{-1} x y^2 \neq w.\end{equation} 
However let $u' := y^2 x y^{-1} x$ and $u'' := y^{-1} x y^2 x$. Then
		\begin{equation} \label{assocCFG3} u' * (u^{-1} * w) = x y^{-1} x^{-1} y^2 x \sim w\end{equation}
and
		\begin{equation} \label{assocCFG4} (w * u^{-1}) * u'' = y^2 x^2 y^{-1} x^{-1} \sim w,\end{equation}
that is if we replace $u$ in the expressions $(\ref{assocCFG1})$ and $(\ref{assocCFG2})$ with some of its cyclic permutations the result is a word which is a cyclic permutation of $w$.\footnote{We observe that $u'$ and $u''$ are the only cyclic permutations of $u$ verifying these properties. However if we take $u = x y^{-1} x y$ and $w = x^2 y^{-1} x^{-1} y$ then we can take as $u'$ the two different cyclic permutations $x y x y^{-1}$ and $y x y^{-1} x$, but as $u''$ we can only take $x y^{-1} x^{-1} y x$.}

We can then conjecture that this situation is general, but it is easy to find a counterexample. Indeed let $w$ be cyclically reduced and let $w = u h f h^{-1}$ for some cyclically reduced word $f$ with $h \neq 1$. Then $u^{-1} * w = f$ and since any cyclic permutation of $u$ has the same length of $u$ then for any cyclic permutation $u'$ of $u$ we have that 
		$$|u' * (u^{-1} * w)| \leq |u| + |f| < |u| + |f| + 2|h| = |w|,$$
so $u * (u^{-1} * w)$ cannot be a cyclic permutation of $w$. We observe however that the last case is special in that $w$ is obtained as a concatenation of $u$ with $h f h^{-1}$ and this concatenation is a cyclically reduced word.

\smallskip

The next theorem proves that the two situations illustrated above are general, i.e., either the first or the second situation hold for any $u$ and $w$. Moreover we will show that we do not need to assume that $u$ and $w$ be cyclically reduced in order for these results to be true.

The analogy of the cyclically reduced product with (\ref{assocFG1}) and (\ref{assocFG2}) is deeper than that. Indeed the identities among relations following from (\ref{assocFG1}) and (\ref{assocFG2}) are $u \centerdot u^{-1} \centerdot w \equiv w$ and $w \centerdot u^{-1} \centerdot u \equiv w$. These identities are strictly basic (Definitions \ref{defBasIAR} and \ref{defBasIAR2}). Let us see if this is true also for (\ref{assocCFG3}) and (\ref{assocCFG4}).

We have that $u' = \rho(y^2 u y^{-2})$, $\rho(u^{-1} w) = y^{-2} x^{-1} y x y^{-1} x^{-1} y^2$, $u^{-1} * w = x^{-1} y x y^{-1} x^{-1} = \rho(y^2 u^{-1} w y^{-2})$, $\rho(u' (u^{-1} * w)) = y^2 x^2 y^{-1} x^{-1} = \rho(y^2 w y^{-2})$, and $\rho(u' (u^{-1} * w)) = u' * (u^{-1} * w)$, so the identity among relations following from (\ref{assocCFG3}) is
	$$y^2 u y^{-2} \centerdot y^2 u^{-1} y^{-2} \centerdot y^2 w y^{-2} \equiv y^2 w y^{-2},$$
which is strictly basic.

In the same way (verification left to the reader) we can prove that the identity among relations following from (\ref{assocCFG4}) is
	$$x^{-1} u x \centerdot x^{-1} u^{-1} x \centerdot x^{-1} w x \equiv x^{-1} w x,$$
again it is strictly basic.

We will prove in the next theorem that this fact too is true for the cyclically reduced product.

\begin{theorem} \label{mainTheor1} Let $u$ and $w$ be words. Then one of the following two cases hold: 
	\begin{enumerate}
		\item there exist words $u', u''$ which are cyclic permutations of $u$ such that
			\begin{equation} \label{MT1crA1} \hat{\rho}(w) \sim u' * (u^{-1} * w) \end{equation} 
		and 	
			\begin{equation} \label{MT1crA2} \hat{\rho}(w) \sim (w * u^{-1}) * u'';\end{equation}

		\item there exist a word $u'$ which is the reduced form of a cyclic permutation of $u$ and a non-empty word $h$ such that
			\begin{equation} \label{MT1crB1} \hat{\rho}(w) \sim u' h (u^{-1}*w) h^{-1}\end{equation}
		and moreover $w * u^{-1} = u^{-1} * w$. In particular this implies that	
			\begin{equation} \label{MT1crB2} \hat{\rho}(w) \sim h (w*u^{-1}) h^{-1} u'. \end{equation}
	\end{enumerate}
Finally the identities among relations involving $u, w, u^{-1}, w^{-1}$ that by Remark \ref{idFrom} follow from (\ref{MT1crA1})-(\ref{MT1crB2}) are strictly basic. 
\end{theorem}

\begin{proof} We observe that we can assume that $u$ and $w$ are reduced in view of (\ref{cpCanc}) and (\ref{u*vNonRed}) of Proposition \ref{summar}.

First we prove that if the identities among relations involving $u, w$ that by Remark \ref{idFrom} follow from (\ref{MT1crA1})-(\ref{MT1crB2}) are basic then they are strictly basic.

Indeed there exist words $\alpha, \beta, \gamma$ such that $u' = \rho(\alpha u \alpha^{-1})$, $u^{-1} * w = \rho(\beta u^{-1} w \beta^{-1})$ and $u' * (u^{-1} * w) = \rho(\gamma u'(u^{-1} * w) \gamma^{-1})$. This implies that $u' * (u^{-1} * w) = \rho(\gamma \alpha u \alpha^{-1} \beta u^{-1} w \beta^{-1} \gamma^{-1})$.

Let (\ref{MT1crA1}) hold; then there exists a word $\delta$ such that $u' * (u^{-1} * w) = \rho(\delta w \delta^{-1})$, thus the identity among relations following from (\ref{MT1crA1}) is 
		$$(\gamma \alpha) u (\alpha^{-1} \gamma^{-1}) \centerdot (\gamma \beta) u^{-1} (\beta^{-1} \gamma^{-1}) \centerdot (\gamma \beta) w (\beta^{-1} \gamma^{-1}) = \delta w \delta^{-1}.$$
If this identity is basic then $\rho(\gamma \alpha) = \rho(\gamma \beta)$ and $\rho(\gamma \beta) = \delta$, thus we have also that $\rho(\gamma \alpha) = \delta$ and the identity is strictly basic.

If (\ref{MT1crA2}) holds the proof is analogous and we omit it. 

Let (\ref{MT1crB1}) hold. Then $u' h (u^{-1}*w) h^{-1} = \rho(\alpha u \alpha^{-1} h \beta u^{-1} w \beta^{-1} h^{-1})$. Moreover there exists a word $\epsilon$ such that $u' h (u^{-1}*w) h^{-1} = \rho(\epsilon w \epsilon^{-1})$, thus the identity among relations following from (\ref{MT1crA1}) is 
		$$\alpha u \alpha^{-1} \centerdot (h \beta) u^{-1} (\beta^{-1} h^{-1}) \centerdot (h \beta) w (\beta^{-1} h^{-1}) = \epsilon w \epsilon^{-1}.$$
If this identity is basic then $\alpha = \rho(h \beta)$ and $\rho(h \beta) = \epsilon$, thus we have also that $\alpha = \epsilon$ and the identity is strictly basic.

The proof for (\ref{MT1crB2}) is analogous.

\smallskip

Now we show that if (\ref{MT1crA1}) holds then (\ref{MT1crA2}) holds; and if (\ref{MT1crB1}) holds and $w * u^{-1} = u^{-1} * w$ then (\ref{MT1crB2}) holds.

\smallskip

Indeed let (\ref{MT1crA1}) hold; then by applying (\ref{MT1crA1}) to $\underline{u}$ and $\underline{w}$ we have that there exists a cyclic permutation $u'$ of $\underline{u}$ such that $\hat{\rho}(\underline{w}) \sim u' * (\underline{u^{-1}} * \underline{w})$ and the identity among relations involving $\underline{u}$ and $\underline{w}$ that follows from this equivalence is basic.

By (\ref{reversesim}) of Proposition \ref{summar} we have that $\underline{\hat{\rho}(\underline{w})} \sim \underline{u' * (\underline{u^{-1}} * \underline{w})}$ and by Remark \ref{idRev} the identity among relations involving $\underline{u}$ and $\underline{w}$ that follows from it is basic. 

By (\ref{revCycRedForm}) of Proposition \ref{summar} we have that $\underline{\hat{\rho}(\underline{w})} = \hat{\rho}(w)$. By (\ref{revCRP}) of Proposition \ref{summar} we have that $\underline{u' * (\underline{u^{-1}} * \underline{w})} = (w * u^{-1}) * \underline{u'}$ and by setting $u'' := \underline{u'}$ we have by (\ref{reversesim}) of Proposition \ref{summar} that $u''$ is a cyclic permutation of $u$. We have thus proved that if (\ref{MT1crA1}) holds then (\ref{MT1crA2}) holds. 

Now let us assume that (\ref{MT1crB1}) holds and that $w * u^{-1} = u^{-1} * w$. Since $u' h (u^{-1}*w) h^{-1} \sim h (u^{-1}*w) h^{-1} u'$, then 
		$$\hat{\rho}(w) \sim h (u^{-1}*w) h^{-1} u' = h (w*u^{-1}) h^{-1} u',$$
proving (\ref{MT1crB2}).

\smallskip

Now we have to prove that either (\ref{MT1crA1}) or (\ref{MT1crB1}) holds and that the identities following from them are basic; moreover we have to prove that if (\ref{MT1crB1}) holds then $w * u^{-1} = u^{-1} * w$.
	
Let us set $f := u^{-1} * w$. We consider the cyclically reduced product of $u^{-1}$ by $w$ and we prove that the claim is true for the three cases of Lemma \ref{shirv}:
	
\begin{enumerate}
	\item there exist words $t, u_1, u_2$ such that $u^{-1} = u_2^{-1} u_1^{-1}$ and $w = u_1 t f t^{-1} u_2$;
		
	\item there exist non-empty words $f_1, f_2$ and words $a, t$ such that $f = f_1 f_2$, $u^{-1} = t f_1 a$ and $w = a^{-1} f_2 t^{-1}$; 
		
	\item there exist words $t, x_1, x_2$ such that $u^{-1} = x_2^{-1} t f t^{-1} x_1^{-1}$ and $w = x_1 x_2$.	
\end{enumerate}
	
\medskip

\textbf{1)} (\ref{MT1crA1}) and (\ref{MT1crB1}) follow from Proposition \ref{tecLemma2c}. By what seen above, also (\ref{MT1crA2}) follows. It remains to prove that if (\ref{MT1crB1}) holds then $u^{-1} * w = w * u^{-1}$.

We have that 
	$$w * u^{-1} = \hat{\rho}(w u^{-1}) = \hat{\rho}(u_1 t f t^{-1} u_2 u_2^{-1} u_1^{-1}) = \hat{\rho}(u_1 t f t^{-1} u_1^{-1}).$$ 

Let $t \neq 1$; then $u_1 t f t^{-1} u_1^{-1}$ is reduced and therefore $w * u^{-1} = f = u^{-1} * w$. Indeed $u_1 t f t^{-1}$ is reduced because it is a subword of $w$; $t^{-1} u_1^{-1}$ is reduced because its inverse is a subword of $w$; finally since $t^{-1} \neq 1$, then $u_1 t f t^{-1} u_1^{-1}$ is reduced by virtue of Remark \ref{uvw}.

Now we prove that if $t = 1$ then (\ref{MT1crA1}) holds. Indeed $w = u_1 f u_2$ and by setting $u' := u_2 u_1$ we have that
	$$u' * (u^{-1} * w) = \hat{\rho}(u' f) = \hat{\rho}(u_2 u_1 f) \sim \hat{\rho}(u_1 f u_2) = \hat{\rho}(w),$$
where the last equivalence follows from Corollary \ref{cycRedPerm}.


\medskip

\textbf{2)} We have that $f_1 = \rho(t^{-1} u^{-1} a^{-1})$ and $f_2 = \rho(a w t)$, so 
	$$f = \rho(t^{-1} u^{-1} a^{-1} a w t) = \rho(t^{-1} u^{-1} w t).$$
We also have that $u = a^{-1} f_1^{-1} t^{-1}$ and the word $u' := t^{-1} a^{-1} f_1^{-1} = \rho(t^{-1} u t)$ is a cyclic permutation of $u$. If we set $\alpha := t^{-1}$ and $\beta := t^{-1}$ then $u' = \rho(\alpha u \alpha^{-1})$ and $f = \rho(\alpha u^{-1} \alpha^{-1} \beta w \beta^{-1})$; thus (\ref{MT1crA1}) and the last part of the claim follow from Lemma \ref{baseLemma}.

\medskip
	
\textbf{3)} We have that $u = x_1 t f^{-1} t^{-1} x_2$ and since $w^{-1} = x_2^{-1} x_1^{-1}$ then
	$$f = \rho(t^{-1} x_2 u^{-1} x_1 t) = \rho(t^{-1} x_2 u^{-1} x_2^{-1} x_1^{-1} w x_1 t).$$ 
	
The word $u' := t^{-1} x_2 x_1 t f^{-1}$ is a cyclic permutation of $u$	and $u' = \rho(t^{-1} x_2 u x_2^{-1} t)$. If we set $\alpha := t^{-1} x_2$ and $\beta := t^{-1} x_1^{-1}$ then $u' = \rho(\alpha u \alpha^{-1})$ and $f = \rho(\alpha u^{-1} \alpha^{-1} \beta w \beta^{-1})$; thus (\ref{MT1crA1}) and the last part of the claim follow from Lemma \ref{baseLemma}.
\end{proof}

\begin{remark} \label{} \rm We show now that the results of Theorem \ref{mainTheor1} cannot be improved.

The example given before Theorem \ref{mainTheor1} shows that in (\ref{MT1crA1}) and (\ref{MT1crA2}) the words $u'$ and $u''$ can be distinct and non-trivial cyclic permutations of $u$. In that example also the cyclic permutations of $w$ are distinct and non-trivial. This shows that (\ref{MT1crA1}) and (\ref{MT1crA2}) cannot be improved.

This is also true when the cyclically reduced product of $u'$ and $u^{-1} * w$ is without cancellation as if $u = x y^2$ and $w = x y x y^3$ (verification left to the reader).

Now let us give an example of $(\ref{MT1crB1})$ and let us also show that $(\ref{MT1crB1})$ cannot be improved. Indeed let $u = x^2 y$ and $w = x y x y x y^{-2} x y$. Then $u^{-1} = y^{-1} x^{-2}$, $u^{-1} * w = x y x y^{-1}$ and since $|w| = 9$, $|u^{-1} * w| = 4$ and $|u| = 3$, then for no cyclic permutation $u'$ of $u$ is $u' * (u^{-1} * w) \sim w$. However let $w' := x y x y x y x y^{-2}$; then $w' \sim w$ and $w' = u' y (u^{-1} * w) y^{-1}$, with $u' = x y x$. Both $w'$ and $u'$ are non-trivial cyclic permutations of $w$ and $u$ respectively.
\end{remark}


\begin{corollary} \label{corToMTh1} Let $u$ and $w$ be words and let us set $f := u^{-1} * w$ and $g := w * u^{-1}$. Then there exist words $u', u'', h$ such that $u'$ and $u''$ are the reduced forms of cyclic permutations of $u$ and such that
	$$\hat{\rho}(w) \sim u' * (h f h^{-1}), \hspace{0.5 cm} \hat{\rho}(w) \sim (h g h^{-1}) * u''.$$
Moreover if $h \neq 1$ then $f = g$, $u' = u''$, $u'*(h f h^{-1}) = u' h f h^{-1}$ and $(h g h^{-1}) * u'' = h g h^{-1} u''$. Finally the identities among relations involving $u, w, u^{-1}, w^{-1}$ that by Remark \ref{idFrom} follow from these equivalences are strictly basic.
\end{corollary}

\begin{proof} Follows trivially from Theorem \ref{mainTheor1}. \end{proof}

\appendix

\section{Identities among relations} \label{secA}

This section deals with identities among relations. Some of the material in this section can also be found in Appendix A of \cite{FirstArticle}, but we have included here in order to make the paper self-contained.

\smallskip

Let $a_1, \dots, a_m, r_1, \dots, r_m, b_1, \dots, b_n, s_1, \dots, s_n$ be words such that the equality 
	\begin{equation} \label{idamre0} \rho(a_1 r_1 a_1^{-1} \dots a_m r_m a_m^{-1}) = \rho(b_1 s_1 b_1^{-1} \dots b_n s_n b_n^{-1})\end{equation} 
holds. Then we say that we have an \textit{identity among relations involving $r_1, \dots, r_m, s_1^{-1}, \dots, s_n^{-1}$} and we denote it 
	\begin{equation} \label{idamre2} a_1 r_1 a_1^{-1} \centerdot \dots \centerdot a_m r_m a_m^{-1} \equiv b_1 s_1 b_1^{-1} \centerdot \dots \centerdot b_n s_n b_n^{-1}.	\end{equation} 
If $n = 0$, that is the right hand side is 1, then we say that the identity is in \textit{normal form}.

\begin{remark} \label{iarConjSome} \rm Let us suppose that (\ref{idamre0}) holds, let $i \in \{1, \dots, m\}$ and let $r'_i = \rho(c r_i c^{-1})$ for some word $c$. Then we have an identity among relations involving $r'_i$, all the $r_j$ except $r_i$ and all the $s_j$. That identity is the same as (\ref{idamre2}) except for the coefficient for $r'_i$ that is $\rho(a_i c^{-1})$.
\end{remark}

Identities among relations are special types of word equations. They arise in the context of group presentations, but we will use them without involving an explicit group presentation. In particular an identity among relations involving $r_1, \dots, r_m$ is an identity among relations for any group presentation having $r_1, \dots, r_m$ as relators. The last claim is obvious if the $r_i$ are basic relators. If some of the $r_i$ are non-basic relators, then the claim follows from Remark \ref{iarRepl}.

\begin{remark} \label{iarRepl} \rm Let us suppose that (\ref{idamre2}) holds and that for some $i$ we have that the reduced form of $r_i$ is equal to the reduced form of $c_1 t_1 c_1^{-1} \dots c_k t_k c_k^{-1}$ for some words $c_1, t_1, \dots c_k, t_k$. Then by replacing in (\ref{idamre2}) the term $a_i r_i a_i^{-1}$ with $d_1 t_1 d_1^{-1} \centerdot \dots \centerdot d_k t_k d_k^{-1}$, where $d_j = a_i c_j$, we obtain an identity among relations involving $s_1, \dots, s_n, t_1, \dots, t_k$ and all the $r_h$ except for $h = i$.
\end{remark}

\begin{definition} \label{}  \rm We say that the identities 
	$$a_2 r_2 a_2^{-1} \centerdot \dots \centerdot a_m r_m a_m^{-1} \equiv a_1 r_1^{-1} a_1^{-1} \centerdot b_1 s_1 b_1^{-1} \centerdot \dots \centerdot b_n s_n b_n^{-1},$$
	$$a_1 r_1 a_1^{-1} \centerdot \dots \centerdot a_{m-1} r_{m-1} a_{m-1}^{-1} \equiv b_1 s_1 b_1^{-1} \centerdot \dots \centerdot b_n s_n b_n^{-1} \centerdot a_m r_m a_m^{-1},$$
	$$b_1 s_1^{-1} b_1^{-1} \centerdot a_1 r_1 a_1^{-1} \centerdot \dots \centerdot a_m r_m a_m^{-1} \equiv b_2 s_2 b_2^{-1} \centerdot \dots \centerdot b_n s_n b_n^{-1}$$
and
	$$a_1 r_1 a_1^{-1} \centerdot \dots \centerdot a_m r_m a_m^{-1} \centerdot b_n s_n b_n^{-1} \equiv b_1 s_1 b_1^{-1} \centerdot \dots \centerdot b_{n-1} s_{n-1} b_{n-1}^{-1}$$
are \textit{1-step equivalent to (\ref{idamre2}}). 

We say that an identity $\iota$ is \textit{equivalent} to an identity $\iota'$ if there exist identities $\iota_1, \dots, \iota_n$ such that $\iota$ is 1-step equivalent to $\iota_1$, $\iota_i$ is 1-step equivalent to $\iota_{i+1}$ for $i \in \{1, \dots, n-1\}$ and $\iota_n$ is 1-step equivalent to $\iota'$.
\end{definition}

\begin{remark} \label{difForIds} \rm An identity among relations can be equivalent to more than one identity in normal form; these are called \textit{the normal forms} of that identity. We prove that two normal forms of the same identity are cyclic permutation one of the other.

The proof is by induction on the number of terms in the right hand side of (\ref{idamre2}), where the claim is obvious when that number is 1. Let that number be $k > 1$ and the claim be true when that number is less than $k$. With (\ref{idamre2}) we can associate the following and only these two identities with $k - 1$ terms on the right,
	\begin{equation} \label{cpId1} a_1 r_1 a_1^{-1} \centerdot \dots \centerdot a_m r_m a_m^{-1}  \centerdot b_k s_k b_k^{-1} \equiv b_1 s_1^{-1} b_1^{-1} \centerdot \dots \centerdot b_{k-1} s_{k-1} b_{k-1}^{-1}\end{equation}
and
	\begin{equation} \label{cpId2} b_1 s_1^{-1} b_1^{-1} \centerdot a_1 r_1 a_1^{-1} \centerdot \dots \centerdot a_m r_m a_m^{-1} \equiv b_2 s_2 b_2^{-1} \centerdot \dots \centerdot b_k s_k b_k^{-1}.\end{equation}
Therefore by induction hypothesis the normal forms of the identity (\ref{idamre2}) are two sets of identities such that the elements in each of these sets are cyclic permutations one of the other. These two sets are the cyclic permutations of the identities (\ref{cpId1}) and (\ref{cpId2}).

It remains to prove that any two elements taken one from the first set and the other from the second set are cyclic permutation one of the other. Since being a cyclic permutation is an equivalence relation, it is enough to prove that one specific element of the first set is a cyclic permutation of one specific element of the second set. This is done by taking the following two elements:
 	$$a_1 r_1 a_1^{-1} \centerdot \dots \centerdot a_m r_m a_m^{-1}  \centerdot b_k s_k^{-1} b_k^{-1} \centerdot \dots b_1 s_1^{-1} b_1^{-1} \equiv 1$$
from the first set and 
	$$b_k s_k^{-1} b_k^{-1} \centerdot \dots b_1 s_1^{-1} b_1^{-1} \centerdot a_1 r_1 a_1^{-1} \centerdot \dots \centerdot a_m r_m a_m^{-1} \equiv 1,$$
from the second set. It is trivial to see that these elements are cyclic permutations one of the other.
\end{remark}

\bigskip

Let $\langle \, X \, | \, R \, \rangle$ be a presentation for a group $G$, with $X$ the set of generators and $R$ that of basic relators. We will assume without loss of generality that $R$ contains the inverse of any of its elements and the reduced form of the cyclic permutations of any of its elements. If $r_1, \dots, r_n \in R$ are such that the identity in normal form
	\begin{equation} \label{idamre} a_1 r_1 a_1^{-1} \centerdot \dots \centerdot a_n r_n a_n^{-1} \equiv 1	\end{equation}
holds, then (\ref{idamre}) determines a product of conjugates of basic relators equal to 1 not only in $G$ but also in $\mathcal{F}(X)$ (we recall that $G$ is a quotient of $\mathcal{F}(X)$).

\bigskip

In order to formalize these notions we introduce some definitions (we will follow \cite{BH}). Let us set $Y := \mathcal{F}(X) \times R$, let us define the inverse of an element $(a, r) \in Y$ as $(a, r^{-1})$ and let us denote $H$ the free monoid on $Y \cup Y^{-1}$. $H$ is the set of finite sequences of elements of $Y$. We denote an element of $H$ as $[(a_1, r_1), \dots, (a_n, r_n)]$, where $a_i \in \mathcal{F}(X)$ and $r_i \in R$. The trivial element of $H$ is the sequence with zero elements.

Let $h := [(a_1, r_1), \dots, (a_n, r_n)] \in H$ and let $(a, r)$, $(b, s)$ be two consecutive elements $(a_i, r_i), (a_{i+1}, r_{i+1})$ of $h$ for some $i \in \{1, \dots, n - 1\}$, in particular $a = a_i$, $r = r_i$, $b = a_{i+1}$, $s = r_{i+1}$. We define the following transformations on $h$ that change it to another element of $H$:
\begin{itemize}
	\item [--] a \textit{Peiffer deletion} deletes in $h$ the elements $(a, r)$, $(b, s)$ if $a = b$ and $r^{-1} = s$;

	
	\item [--] an \textit{exchange} replaces in $h$ the pair of elements $(a, r)$, $(b, s)$  either with the pair
		$$(b, s), (\rho(b s^{-1} b^{-1} a), r)$$
	(we call it an \textit{exchange of type A at the $i$-th position} or \textit{exchange of type A-$i$}) or with the pair
		$$(\rho(a r a^{-1} b), s), (a, r)$$
	(we call it an \textit{exchange of type B at the $i$-th position} or \textit{exchange of type B-$i$})).	
\end{itemize}
Peiffer deletions and exchanges leave unchanged the $(a_j, r_j)$ for $j \neq i, i+1$.




Given two elements $h_1, h_2 \in H$, we say that \textit{$h_1$ Peiffer collapses to $h_2$} if $h_2$ can be obtained from $h_1$ by applying Peiffer deletions and exchanges. 

There is a bijection $\chi$ between $H$ and the set of products of conjugates of elements of $R$ given by associating the element $h = [(a_1, r_1), \dots, (a_n, r_n)] \in H$ with the following product of conjugates of elements of $R$,
	$$a_1 r_1 a_1^{-1} \centerdot \dots \centerdot a_n r_n a_n^{-1}.$$
Also we define a monoid homomorphism $\psi$ from $H$ to $\mathcal{F}(X)$ by $\psi(h) := \rho(a_1 r_1 a_1^{-1} \dots a_n r_n a_n^{-1})$. If $\psi(h) = 1$, that is if $h$ belongs to the kernel of $\psi$, then we say that $h$ determines the identity among relations in normal form (\ref{idamre}). We say that this identity among relations \textit{Peiffer collapses to 1} if $h$ Peiffer collapses to the trivial element of $H$.

The restriction of $\chi$ to the kernel of $\psi$ determines a bijection with the set of identities among relations in normal form involving elements of $R$.


\begin{remark} \label{iarPeiffCons} \rm We have seen in the introduction to this section that if $r_1$, $\dots$, $r_n$ are relators of a group presentation $\mathcal{P} := \langle \, X \, | \, S \, \rangle$ then an identity among relations involving $r_1, \dots, r_n$ determines an identity among relations for $\mathcal{P}$, that is an identity involving the basic relators of $\mathcal{P}$.
	
By virtue of the Corollary at page 159 of \cite{BH} we have also that if the identity involving $r_1, \dots, r_n$ Peiffer collapses to 1 then also the identity involving basic relators determined by it Peiffer collapses to 1.
\end{remark}

\begin{definition} \label{defBasIAR} \rm An identity among relations in normal form is said \textit{basic} if the corresponding element of $H$ collapses to 1 by means of only Peiffer deletions. Let $h := [(a_1, r_1), \dots, (a_n, r_n)]$ be that element; if moreover $a_1 = a_2 = \dots = a_n$ then that identity among relations is said \textit{strictly basic}.
\end{definition}

\noindent \textbf{Example.} The identity among relations $a r a^{-1} \centerdot b s b^{-1} \centerdot b s^{-1} b^{-1} \centerdot a r^{-1} a^{-1} \equiv 1$ is basic. If moreover $a = b$ it is strictly basic.
\medskip

\begin{remark} \label{equCondBas} \rm As before we denote $Y$ the set $\mathcal{F}(X) \times R$ and let us consider the free group on $Y$. Since $H$ is the free monoid on $Y \cup Y^{-1}$, then an identity among relations is basic if and only if the corresponding element of $H$ is 1 in the free group on $Y$.
\end{remark}

\begin{remark} \label{equCondBas2} \rm We now show that if one normal form of an identity is basic, then all the normal forms of that identity are basic.

Indeed by Remark \ref{equCondBas}, an identity among relations is basic if and only if the corresponding element of $\mathcal{M}(Y \cup Y^{-1})$ reduces to 1 in the free group on $Y$. By Remark \ref{difForIds}, two normal forms of the same identity are cyclic permutations one of the other, therefore if one of them reduces to 1 in the free group, all its cyclic permutations reduce to 1 too because a cyclic permutation is a special case of conjugation and in a group the only conjugate to 1 is 1 itself.
\end{remark}

\begin{definition} \label{defBasIAR2} \rm An identity among relations is said \textit{(strictly) basic} if one (and by Remark \ref{equCondBas2} all) of its normal forms is (strictly) basic.
\end{definition}

\begin{remark} \label{iarBasicConj} \rm Let us given a basic identity among relations involving some words and let us consider the identity involving the same words where some of these words are replaced by conjugations as seen in Remark \ref{iarConjSome}. Then it is obvious that the new identity among relations is basic too.
\end{remark}

\begin{remark} \label{basicIAR} \rm Let $a_1, \dots, a_n, r_1, \dots, r_n$ be words and let $k < n$; then it is easy to see that the following two identities hold and that they are basic:  
	$$a_k r_k a_k^{-1} \centerdot \dots \centerdot a_n r_n a_n^{-1} \equiv a_{k-1} r_{k-1}^{-1} a_{k-1}^{-1} \centerdot \dots \centerdot  a_1 r_1^{-1} a_1^{-1} \centerdot a_1 r_1 a_1^{-1} \centerdot \dots \centerdot a_n r_n a_n^{-1}$$
and
	$$a_1 r_1 a_1^{-1} \centerdot \dots \centerdot a_k r_k a_k^{-1} \equiv  a_1 r_1 a_1^{-1} \centerdot \dots \centerdot a_n r_n a_n^{-1} \centerdot a_n r_n^{-1} a_n^{-1} \centerdot \dots a_{k+1} r_{k+1}^{-1} a_{k+1}^{-1}.$$
\end{remark}

\begin{remark} \label{pcrFrom} \rm Let $R$ be a set of reduced words, that is $R \subset \mathcal{F}(X)$. Let us consider the following operations on the elements of $\mathcal{F}(X)$: reduced product, cyclically reduced product, cyclically reduced form, conjugations, reduced form of cyclic permutations. 

Let $\mathcal{N}$ be the normal closure of $R$ in $\mathcal{F}(X)$; then $\mathcal{N}$ is the subset of $\mathcal{F}(X)$ generated by $R$ and by the above operations. Indeed cyclic permutations and the cyclically reduced form are special cases of conjugations and the cyclically reduced product is obtained by composing the cyclically reduced form with the reduced product.

Let $\sigma$ be a sequence of the above listed operations on the elements of $R$ and let $u \in \mathcal{N}$ be the result of $\sigma$. We will show how to associate with $\sigma$ an element $[(a_1, r_1), \cdots, (a_n, r_n)]$ of $H$ with the property that
	$$\rho(a_1 r_1 a_1^{-1} \cdots a_n r_n a_n^{-1}) = u.$$

\smallskip

- Let us take a sequence of length one. This is an element $r$ of $R$ and we associate with it the element $[(1, r)]$ of $H$. 

We can suppose by induction hypothesis that there is a natural number $k$ such that for each sequence $\sigma$ of length less than $k$ we have associated with $\sigma$ an element of $H$ with the properties specified above.

- Let us given sequences $\sigma, \sigma'$ of length less than $k$ with results respectively $u$ and $u'$. Then by induction hypothesis there exist $r_1, \cdots, r_m$, $s_1, \cdots, s_n \in R$ and $a_1, \cdots, a_m$, $b_1, \cdots, b_n \in \mathcal{F}(X)$ such that we have associated with $\sigma$ an element $[(a_1, r_1), \cdots, (a_m, r_m)] \in H$ such that 
	$$u = \rho(a_1 r_1 a_1^{-1} \cdots a_m r_m a_m^{-1})$$ 
and with $\sigma'$ an element $[(b_1, s_1), \cdots, (b_n, s_n)] \in H$ such that 
	$$u' =\rho(b_1 s_1 b_1^{-1} \cdots b_n s_n b_n^{-1}).$$
Let us consider the sequence $\tau$ having all the operations of $\sigma$ and $\sigma'$ plus the reduced product of $u$ by $u'$. Then we associate with $\tau$ the element
	$$[(a_1, r_1), \cdots, (a_m, r_m), (b_1, s_1), \cdots, (b_n, s_n)] \in H;$$
obviously $\rho(a_1 r_1 a_1^{-1} \cdots a_m r_m a_m^{-1} b_1 s_1 b_1^{-1} \cdots b_n s_n b_n^{-1}) = \rho(u u')$.

- Now let us consider a sequence $\sigma_1$ having all the operations of $\sigma$ plus the conjugation of $u$ by a word $b$. Then we associate with $\sigma_1$ the element $[(c_1, r_1), \cdots, (c_m, r_m)] \in H$ where $c_i = \rho(b a_i)$. Obviously 
	$$\rho(c_1 r_1 c_1^{-1} \cdots c_m r_m c_m^{-1}) = \rho(b u b^{-1}).$$

- Now let us consider a sequence $\sigma_2$ having all the operations of $\sigma$ plus the cyclically reduced form of $u$. The previous cases show how to associate with $\sigma_2$ an element of $H$ with the above properties because by virtue of (\ref{scope}) of Proposition \ref{summar} the cyclically reduced form is a special case of conjugation.

-Now let us consider a sequence $\sigma_3$ having all the operations of $\sigma$ plus the reduced form of a cyclic permutation of $u$. This means that there exist words $u_1, u_2$ such that $u = u_1 u_2$ and that the last operation of $\sigma_3$ is the conjugation of $u$ by either $u_2$ or by $u_1^{-1}$. This implies that we associate with $\sigma_3$ the element $[(c_1, r_1), \cdots, (c_m, r_m)] \in H$ where $c_i$ for $i = 1, \cdots, m$ can be either equal to $\rho(u_2 a_i)$ or to $\rho(u_1^{-1} a_i)$.

- Finally if $\tau$ is the sequence having all the operations of $\sigma$ and $\sigma'$ plus the cyclically reduced product of $u$ by $u'$, then the previous cases show how to associate with $\tau$ an element of $H$ with the properties stated above because the cyclically reduced product is the composition of the reduced product with the cyclically reduced form.
\end{remark}

\begin{remark} \label{seqFromProd} \rm We show how to associate with a product of conjugates of elements of $R$ a sequence of operations on $R$ as described in Remark \ref{pcrFrom}.

Indeed with $a_1 r_1 a_1^{-1} \centerdot \dots \centerdot a_m r_m a_m^{-1}$ we associate the following sequence: conjugation of $r_1$ with $a_1$; conjugation of $r_2$ with $a_2$; $\dots$; conjugation of $r_m$ with $a_m$; reduced product of $a_1 r_1 a_1^{-1}$ by $a_2 r_2 a_2^{-1}$; reduced product of $a_1 r_1 a_1^{-1} a_2 r_2 a_2^{-1}$ by $a_3 r_3 a_3^{-1}$; $\dots$; reduced product of $a_1 r_1 a_1^{-1} \dots  a_{m-1} r_{m-1} a_{m-1}^{-1}$ by $a_m r_m a_m^{-1}$.

In particular, given words $u$ and $v$, we associate with $u*v$ the product $\alpha u \alpha^{-1} \centerdot \alpha v \alpha^{-1}$, where $\alpha$ is such that $u*v = \rho(\alpha u v \alpha^{-1})$ (see (\ref{scope}) of Proposition \ref{summar}). \end{remark}

\begin{remark} \label{idFrom} \rm Let $u, u' \in \mathcal{N}$ be obtained respectively from sequences $\sigma$ and $\sigma'$ of operations on $R$ as described in Remark \ref{pcrFrom}, in particular in view of Remark \ref{seqFromProd} let $u, u'$ be the reduced forms of products of conjugates of elements of $R$. Let us suppose that $u \sim u'$; we show how to associate with $\sigma$, $\sigma'$ and the equivalence $u \sim u'$ an identity among relations involving elements of $R$.	
	
Indeed the procedure described in Remark \ref{pcrFrom} associates with $\sigma$ and $\sigma'$ elements $h := [(a_1, r_1), \dots, (a_m, r_m)]$ and $h' := [(b_1, s_1), \dots, (b_n, s_n)]$ of $H$ such that $\rho(a_1 r_1 a_1^{-1} \dots a_m r_m a_m^{-1}) = u$ and $\rho(b_1 s_1 b_1^{-1} \dots b_n s_n b_n^{-1}) = u'$.

If $u \sim v$ then $u$ and $v$ are conjugates and thus there exists a word $c$ such that $u = \rho(c v c^{-1})$. We associate with $\sigma$, $\sigma'$ and the equivalence $u \sim v$ the following identity among relations 
	$$a_1 r_1 a_1^{-1} \centerdot \dots \centerdot a_m r_m a_m^{-1} \equiv d_1 s_1 d_1^{-1} \centerdot \dots \centerdot d_n s_n d_n^{-1}$$
where $d_i = c b_i$.	
\end{remark}

\begin{remark} \label{idFrom2} \rm Let $R \subset \mathcal{F}(X)$, let $\sigma$ and $\sigma'$ be sequences of operations on $R$ as described in Remark \ref{pcrFrom}, let $w$ and $w'$ be the results of $\sigma$ and $\sigma'$ respectively and let $w \sim w'$. 
	
Let us suppose that $w \sim w'$ and that the identity among relations that by Remark \ref{idFrom} follows from this equivalence is 
	\begin{equation} \label{idFrom2Eq1} a_1 r_1 a_1^{-1} \centerdot \dots \centerdot a_m r_m a_m^{-1} \equiv a_{m+1} r_{m+1} a_{m+1}^{-1} \centerdot \dots \centerdot a_n r_n a_n^{-1}.
	\end{equation}

Let $\sigma''$ be a sequence of operations on $R$ that has all the operations of $\sigma'$ plus a cyclic permutation and let $w''$ be the result of $\sigma''$. This implies that $w' \sim w''$ and by transitivity that $w \sim w''$. We prove that the identity among relations that by Remark \ref{idFrom} follows from $\sigma''$ is
	\begin{equation} \label{idFrom2Eq2} b_1 r_1 b_1^{-1} \centerdot \dots \centerdot b_m r_m b_m^{-1} \equiv b_{m+1} r_{m+1} b_{m+1}^{-1} \centerdot \dots \centerdot b_n r_n b_n^{-1},\end{equation}	
where $b_1, \dots, b_n$ are words such that $b_k = b a^{-1} a_k$ for some words $a, b$ and for $k = 1, \dots, n$.

Indeed by the proof of Remark \ref{idFrom}	we have that there exists a word $a$ such that the products of conjugates of elements of $R$ associated with $\sigma$ and $\sigma'$ are $a_1 r_1 a_1^{-1} \centerdot \dots \centerdot a_m r_m a_m^{-1}$ and $c_{m+1} r_{m+1} c_{m+1}^{-1} \centerdot \dots \centerdot c_n r_n c_n^{-1}$ respectively, where $c_j = \rho(a^{-1} a_j)$ for $j = m+1, \dots, n$.
	
From Remark \ref{pcrFrom} we have that there exists a word $b$ such that the product of conjugates of elements of $R$ associated with $\sigma''$ is $b_{m+1} r_{m+1} b_{m+1}^{-1} \centerdot \dots \centerdot b_n r_n b_n^{-1}$, where $b_j = \rho(b c_j) = \rho(b a^{-1} a_j)$ for $j = m+1, \dots, n$.

This means that $w' = \rho(a^{-1} w a)$, $w'' = \rho(b w' b^{-1})$, therefore  $w'' = \rho(b a^{-1} w a b^{-1})$ and by setting $b_i : = b a^{-1} a_i$ for $i = 1, \dots, m$ we have that the identity among relations that by Remark \ref{idFrom} follows from the equivalence $w \sim w''$ is
	$$b_1 r_1 b_1^{-1} \centerdot \dots \centerdot b_m r_m b_m^{-1} \equiv b_{m+1} r_{m+1} b_{m+1}^{-1} \centerdot \dots \centerdot b_n r_n b_n^{-1},$$
where $b_k = b a^{-1} a_k$ for $k = 1, \dots, n$.
\end{remark}

\begin{remark} \label{basicToo} \rm If (\ref{idFrom2Eq1}) is basic then (\ref{idFrom2Eq2}) is basic too. Indeed by Remark \ref{equCondBas2} if (\ref{idFrom2Eq1}) is basic then any normal form of (\ref{idFrom2Eq1}) corresponds to 1 in the free group on $Y$. We have that also any normal form of (\ref{idFrom2Eq2}) corresponds to 1 in the free group on $Y$, because $b_k = b a^{-1} a_k$ for every $k$, so if for some $h, k$ we have that $a_h = a_k$ then $b_h = b_k$.
\end{remark}

\begin{remark} \label{idRev} \rm Let $R \subset \mathcal{F}(X)$ and let $\sigma$ be a sequence of operations on $R$ as described in Remark \ref{pcrFrom}. We define the \textit{reverse} of $\sigma$, denoted $\underline{\sigma}$, by taking the reverse of each operation of $\sigma$. With $\underline{\sigma}$ we can associate a product of conjugates of reverses of elements of $R$ that is the reverse of that associated with $\sigma$. In particular if 
	$$a_1 r_1 a_1^{-1} \centerdot \dots \centerdot a_m r_m a_m^{-1}$$ 
is the product of conjugates of elements of $R$ associated with $\sigma$ then the one associated with $\underline{\sigma}$ is 
	$$\underline{a_m} \, \underline{r_m} \,\underline{a_m^{-1}} \centerdot \dots \centerdot \underline{a_1} \, \underline{r_1} \, \underline{a_1^{-1}}.$$
In particular the result of $\underline{\sigma}$ is the reverse of the result of $\sigma$.
	
Now let $\sigma$ and $\sigma'$ be sequences of operations on $R$, let $w$ and $w'$ be the results of $\sigma$ and $\sigma'$ respectively and let $w \sim w'$. We have that $\underline{w}$ and $\underline{w'}$ are the results of $\underline{\sigma}$ and $\underline{\sigma'}$, that $\underline{w} \sim \underline{w'}$ and that the identity among relations that follows from this equivalence is the reverse of the one that follows from $w \sim w'$ as shown in Remark \ref{idFrom2}. In particular the first identity is basic in $r_1,  \dots, r_m$ if and only if the second is basic in $\underline{r_1},  \dots, \underline{r_m}$.	
\end{remark}

\section{Some technical results} \label{secB}

We recall that \textit{basic} identities among relations have been defined in Definitions \ref{defBasIAR} and \ref{defBasIAR2}.

\begin{lemma} \label{iarAppB} Let $w, h, f, u$ be words such that  
	\begin{equation} \label{iarAppBHyp1} w \sim \rho(h f h^{-1} u) \end{equation} 
or	
	\begin{equation} \label{iarAppBHyp1A} w \sim \rho(u h f h^{-1}) \end{equation} 	
holds. Then there exists a word $b$ such that $f = \rho(h^{-1} b^{-1} w b u^{-1} h)$ and the identity among relations involving $u$, $w$, $u^{-1}$, $w^{-1}$ that by Remark \ref{idFrom} follows from (\ref{iarAppBHyp1}) or (\ref{iarAppBHyp1A}) is basic. \end{lemma}

\begin{proof} We give the proof of (\ref{iarAppBHyp1}), that of (\ref{iarAppBHyp1A}) being analogous.
	
(\ref{iarAppBHyp1}) implies that there exist words $b, c$ such that $w = bc$ and $\rho(h f h^{-1} u) = c b$, implying that 
	\begin{equation} \label{iarAppBHyp2} w = \rho(b h f h^{-1} u b^{-1}). \end{equation}
From (\ref{iarAppBHyp2}) we have that $f = \rho(h^{-1} b^{-1} w b u^{-1} h)$, which implies the first part of the claim, so we have the equality
	$$w = \rho(b h h^{-1} b^{-1} w b u^{-1} h h^{-1} u b^{-1}) = \rho(w b u^{-1} u b^{-1}),$$
which by Remark \ref{basicIAR} implies a basic identity among relations.\end{proof}

\begin{lemma} \label{baseLemma} Let $u, w, \alpha, \beta$ be words. Let us set $u_0 := \rho(\alpha u \alpha^{-1})$ and either $f := \rho(\alpha u^{-1} \alpha^{-1} \beta w \beta^{-1})$ or $f := \rho(\beta w \beta^{-1} \alpha u^{-1} \alpha^{-1})$. Then
	\begin{equation} \label{baLem1} \hat{\rho}(w) \sim u_0 * f\end{equation}	
and the identity among relations involving $u, w, u^{-1}, w^{-1}$ that by Remark \ref{idFrom} follows from (\ref{baLem1}) is basic.
	
Now let $\beta w \beta^{-1}$ be reduced; if $f = \rho(\alpha u^{-1} \alpha^{-1} \beta w \beta^{-1})$ then (\ref{baLem1}) can be replaced by the stronger result 
	\begin{equation} \label{baLem3} \hat{\rho}(w) = u_0 * f;	
	\end{equation}	
if $f := \rho(\beta w \beta^{-1} \alpha u^{-1} \alpha^{-1})$ then (\ref{baLem1}) can be replaced by 
	\begin{equation} \label{baLem4} \hat{\rho}(w) = f * u_0.	
	\end{equation}
\end{lemma}


\begin{proof} By (\ref{scope}) of Proposition \ref{summar} there exist words $\gamma$ and $\delta$ such that 
	\begin{equation} \label{baLemEq1} u_0*f = \hat{\rho}(u_0 f) = \rho(\gamma u_0 f \gamma^{-1})
	\end{equation}
and
	\begin{equation} \label{baLemEq2} \hat{\rho}(w) = \rho(\delta w \delta^{-1}).	
	\end{equation}
Let $f = \rho(\alpha u^{-1} \alpha^{-1} \beta w \beta^{-1})$; then we have that 
	\begin{equation} \label{baLemEq3} \rho(u_0 f) = \rho(\alpha u \alpha^{-1} \alpha u^{-1} \alpha^{-1} \beta w \beta^{-1}) = \rho(\beta w \beta^{-1}),
	\end{equation}
therefore
	$$u_0 * f = \hat{\rho}(u_0 f) = \hat{\rho}(\beta w \beta^{-1}) \sim \hat{\rho}(w)$$
by virtue of Corollary \ref{permCycRedForm}, proving (\ref{baLem1}). If moreover $\beta w \beta^{-1}$ is reduced then $\hat{\rho}(\beta w \beta^{-1}) = \hat{\rho}(w)$ by Corollary \ref{permCycRedForm}, proving (\ref{baLem3}).

Let $f = \rho(\beta w \beta^{-1} \alpha u^{-1} \alpha^{-1})$; then we have that 
	\begin{equation} \label{baLemEq3'} \rho(f u_0) = \rho(\beta w \beta^{-1} \alpha u \alpha^{-1} \alpha u^{-1} \alpha^{-1}) = \rho(\beta w \beta^{-1}),
	\end{equation}
therefore
	$$u_0 * f \sim f * u_0 = \hat{\rho}(f u_0) = \hat{\rho}(\beta w \beta^{-1}) \sim \hat{\rho}(w)$$
by virtue of Corollary \ref{permCycRedForm}, proving (\ref{baLem1}). If moreover $\beta w \beta^{-1}$ is reduced then $\hat{\rho}(\beta w \beta^{-1}) = \hat{\rho}(w)$ by Corollary \ref{permCycRedForm} and this implies (\ref{baLem4}).

By (\ref{baLem1}) there exist words $x, y$ such that $\hat{\rho}(u_0 f) = xy$ and $\hat{\rho}(w) = yx$. This implies that $\hat{\rho}(u_0 f) = \rho(x \hat{\rho}(w) x^{-1})$, that together with (\ref{baLemEq2}) gives
	\begin{equation} \label{baLemEq5} \hat{\rho}(u_0 f) = \rho(x \delta w \delta^{-1} x^{-1})\end{equation}
and applying (\ref{baLemEq5}) to (\ref{baLemEq1}) we have that
	$$\rho(\gamma u_0 f \gamma^{-1}) = \rho(x \delta w \delta^{-1} x^{-1}).$$
The last equality together with (\ref{baLemEq3}) gives $\rho(\gamma \beta w \beta^{-1} \gamma^{-1}) = \rho(x \delta w \delta^{-1} x^{-1})$.

By Remark \ref{commFG} there exist $c \in \mathcal{F}(X)$ and $m, n \in \mathbb{Z}$ such that $w = \rho(c^m)$ and $\rho(x \delta) = \rho(\gamma \beta c^n)$, which implies that $\delta = \rho(x^{-1} \gamma \beta c^n)$. We show that we can assume that $n = 0$, that is $\rho(x \delta) = \rho(\gamma \beta)$. Indeed let $\epsilon := \rho(x^{-1} \gamma \beta)$, that is $\delta = \rho(\epsilon c^n)$; then 
	$$\rho(\delta w \delta^{-1}) = \rho(\epsilon c^n c^m c^{-n} \epsilon^{-1}) = \rho(\epsilon c^m \epsilon^{-1}) = \rho(\epsilon w \epsilon^{-1}),$$
so in (\ref{baLemEq2}) we could replace $\delta$ with $\epsilon$.

From (\ref{baLemEq1}) and the definitions of $u_0$ and $f$ we have that $u_0*f$ is the reduced form of 
	$$(\gamma \alpha) u (\alpha^{-1} \gamma^{-1}) \centerdot (\gamma \alpha) u^{-1} (\alpha^{-1} \gamma^{-1}) \centerdot (\gamma \beta) w (\beta^{-1} \gamma^{-1}).$$
The latter together with (\ref{baLemEq5}) implies that the identity among relations which by Remark \ref{basicIAR} follows from (\ref{baLem1}) is
	$$(\gamma \alpha) u (\alpha^{-1} \gamma^{-1}) \centerdot (\gamma \alpha) u^{-1} (\alpha^{-1} \gamma^{-1}) \centerdot (\gamma \beta) w (\beta^{-1} \gamma^{-1}) \equiv (x \delta) w (\delta^{-1} x^{-1}),$$
which is basic since by what seen above $\rho(x \delta) = \rho(\gamma \beta)$.
\end{proof}


\begin{proposition} \label{tecLemma2} Let $f, t, u$ be words such that $t f t^{-1} u$ is a reduced word. Let us set $v := \hat{\rho}(t f t^{-1} u)$. Then there exists a cyclic permutation $u_0$ of $u$ such that either
	\begin{equation} \label{u_0*fSim} v \sim u_0*f \end{equation}
	(in particular $v = u_0*f$ or $v = f*u_0$)	or there exists a non-empty word $h$ such that 
	\begin{equation} \label{hfh^{-1}u_0} v = h f h^{-1} u_0.\end{equation} 	

Finally the identities among relations involving $u, v, u^{-1}, v^{-1}$ that by Remark \ref{idFrom} follow from (\ref{u_0*fSim}) and (\ref{hfh^{-1}u_0}) are basic.
\end{proposition}

\begin{proof} Since $t f t^{-1} u$ is a reduced word and $v = \hat{\rho}(t f t^{-1} u)$ then by (\ref{scope}) of Proposition \ref{summar} there exists a word $s$ such that $t f t^{-1} u = s v s^{-1}$; we will prove the claim by analyzing all possible cases of this word equation. We observe that since $v$ is cyclically reduced then $\hat{\rho}(v) = v$.

Cases 1-6 will be proved by using Lemma \ref{baseLemma}. In particular it is enough to prove that there exist a cyclic permutation $u_0$ of $u$ and words $\alpha$ and $\beta$ such that: 	
\begin{enumerate} [(I)]	
	\item \label{rtl2-O} $u_0 = \rho(\alpha u \alpha^{-1})$;
		
	\item \label{rtl2-I} $f = \rho(\alpha u^{-1} \alpha^{-1} \beta v \beta^{-1})$ or $f = \rho(\beta v \beta^{-1} \alpha u^{-1} \alpha^{-1})$;
		
	\item \label{rtl2-II} $\beta v \beta^{-1}$ is reduced.	
\end{enumerate}

	

\medskip	
	
\textbf{1.}  \hspace{0.1cm}
	\begin{tabular}{|c|c|c|}
		\hline
		\rule{0pt}{2.3ex}
		$\,\,\,\,\,\,\,\,\,\,\,\,\,\,\,  s \,\,\,\,\,\,\,\,\,\,\,\,\,\,$ & $v$ & $s^{-1}$ \\  
		\hline
	\end{tabular}
	
	\hspace{0.66cm}
	\begin{tabular}{|c|c|c|c|}
		\hline 
		\rule{0pt}{2.3ex}
		$t$ & $f$ & $t^{-1}$ & $\,\,\,\,\,\,\,\,\,\, u \,\,\,\,\,\,\,\,\,\,$ \\
		\hline
	\end{tabular}	
	
\medskip	
	
\noindent There exists a word $u_1$ such that $s = t f t^{-1} u_1$ and $u = u_1 v s^{-1}$, so $s^{-1} = u_1^{-1} t f^{-1} t^{-1}$, therefore $u = u_1 v u_1^{-1} t f^{-1} t^{-1}$ from which we derive that $f = \rho(t^{-1} u^{-1} u_1 v u_1^{-1} t)$, so $f = \rho(\alpha u^{-1} \alpha^{-1} \beta v \beta^{-1})$ with $\alpha = t^{-1}$ and $\beta = t^{-1} u_1$, proving (\ref{rtl2-I}).
	
The word $u_0:= t^{-1} u_1 v u_1^{-1} t f^{-1} = \rho(t^{-1} u t)$ is a cyclic permutation of $u$ and $u_0 = \rho(\alpha u \alpha^{-1})$, proving (\ref{rtl2-O}).
	
It remains to prove (\ref{rtl2-II}), i.e., that is $t^{-1} u_1 v u_1^{-1} t$ is a reduced word. This is indeed the case because $t f t^{-1} u$ is a reduced word and 
	$$t f t^{-1} u = t f t^{-1} u_1 v u_1^{-1} t f^{-1} t^{-1},$$
so $t^{-1} u_1 v u_1^{-1} t$ is a subword of it.

\bigskip

\textbf{2.} \hspace{0.1cm}
	\begin{tabular}{|c|c|c|}
		\hline
		\rule{0pt}{2.3ex}
		$\,\,\,\,\,\,\,\, s \,\,\,\,\,\,\,\,\,$ & $\,\,\, v \,\,\,$ & $s^{-1}$ \\  
		\hline
	\end{tabular}
	
	\hspace{0.66cm}
	\begin{tabular}{|c|c|c|c|}
		\hline 
		\rule{0pt}{2.3ex}
		$t$ & $f$ & $t^{-1}$ & $\,\,\,\,\,\,\, u \,\,\,\,\,\,\,$ \\
		\hline
	\end{tabular}	
	
\medskip	
	
\noindent There exist words $t_1, v_1, v_2$ such that $s = t f t_1^{-1}$, $t^{-1} = t_1^{-1} v_1$, $v = v_1 v_2$, $u = v_2 s^{-1}$. So $t = v_1^{-1} t_1$ and $s^{-1} = t_1 f^{-1} t^{-1} = t_1 f^{-1} t_1^{-1} v_1$, thus $u = v_2 t_1 f^{-1} t_1^{-1} v_1$ and therefore $u^{-1} = v_1^{-1} t_1 f t_1^{-1} v_2^{-1}$, which implies that 
	$$f = \rho(t_1^{-1} v_1 u^{-1} v_2 t_1) = \rho(t_1^{-1} v_1 v_2 v_2^{-1} u^{-1} v_2 t_1) =$$ 
	$$\rho(t_1^{-1} v v_2^{-1} u^{-1} v_2 t_1) = \rho(\beta v \beta^{-1} \alpha u^{-1} \alpha^{-1}),$$ 
with $\alpha = t_1^{-1} v_2^{-1}$ and $\beta = t_1^{-1}$, proving (\ref{rtl2-I}).
	
The word $u_0:= f^{-1} t_1^{-1} v_1 v_2 t_1 = \rho(t_1^{-1} v_2^{-1} u v_2 t_1)$ is a cyclic permutation of $u$ and $u_0 = \rho(\alpha u^{-1} \alpha^{-1})$, proving (\ref{rtl2-O}).

It remains to prove (\ref{rtl2-II}), i.e., that $t_1^{-1} v t_1$ is a reduced word. This is indeed the case because $s v s^{-1}$ is a reduced word and $s v s^{-1} = t f t_1^{-1} v t_1 f^{-1} t^{-1}$, so $t_1^{-1} v t_1$ is a subword of it.

\bigskip

\textbf{3.} \hspace{0.1cm}
	\begin{tabular}{|c|c|c|}
		\hline
		\rule{0pt}{2.3ex}
		$\,\,\,\,\,\,\,  s \,\,\,\,\,$ & $v$ & $s^{-1}$ \\  
		\hline
	\end{tabular}
	
	\hspace{0.66cm}
	\begin{tabular}{|c|c|c|c|}
		\hline 
		\rule{0pt}{2.3ex}
		$t$ & $f$ & $\,\, t^{-1} \,$ & $u$ \\
		\hline
	\end{tabular}	
	
\medskip	
	
\noindent There exist words $t_1, t_2$ such that $s = t f t_2^{-1}$, $t^{-1} = t_2^{-1} v t_1^{-1}$, $s^{-1} = t_1^{-1} u$. So $s^{-1} = t_2 f^{-1} t^{-1} = t_2 f^{-1} t_2^{-1} v t_1^{-1}$ and since $s^{-1} = t_1^{-1} u$ then $t_2 f^{-1} t_2^{-1} v t_1^{-1} = t_1^{-1} u$, therefore by (\ref{remCycPerm}) of Proposition \ref{summar} $t_2 f^{-1} t_2^{-1} v$ is a cyclic permutation of $u$.

The word $u_0 := f^{-1} t_2^{-1} v t_2$ is a cyclic permutation of $t_2 f^{-1} t_2^{-1} v$ and thus of $u$, then there exists a word $\alpha$ such that $u_0 = \rho(\alpha u \alpha^{-1})$, proving (\ref{rtl2-O}). We have that $f = \rho(t_2^{-1} v t_2 u_0^{-1})$, so by setting $\beta := t_2^{-1}$ we have that $f = \rho(\beta v \beta^{-1} \alpha u^{-1} \alpha^{-1})$, proving (\ref{rtl2-I}).

It remains to prove (\ref{rtl2-II}), i.e., that $t_2^{-1} v t_2$ is a reduced word. This is indeed the case because $s v s^{-1}$ is a reduced word and $s v s^{-1} = t f t_2^{-1} v t_2 f^{-1} t^{-1}$, so $t_2^{-1} v t_2$ is a subword of it.

\bigskip

\textbf{4.} \hspace{0.1cm}
	\begin{tabular}{|c|c|c|}
		\hline
		\rule{0pt}{2.3ex}
		$\,\, s \,\,\, $ & $\,\,\,\,\,\, v \,\,\,\,\,\,$ & $s^{-1}$ \\  
		\hline
	\end{tabular}
	
	\hspace{0.66cm}
	\begin{tabular}{|c|c|c|c|}
		\hline 
		\rule{0pt}{2.3ex}
		$t$ & $f$ & $t^{-1}$ & $\,\,\,\, u \,\,\,\,$ \\
		\hline
	\end{tabular}	
	
\medskip	
	
\noindent There exist words $f_1, f_2, u_1$ such that $s = t f_1$, $f = f_1 f_2$, $v = f_2 t^{-1} u_1$, $u = u_1 s^{-1}$. So $s^{-1} = f_1^{-1} t^{-1}$ and thus $u = u_1 f_1^{-1} t^{-1}$. This implies that $u^{-1} = t f_1 u_1^{-1}$ and then $f_1 = \rho(t^{-1} u^{-1} u_1)$. Moreover $f_2 = \rho(v u_1^{-1} t)$, therefore since $\rho(t^{-1} u^{-1} u_1) = f_1$ then
	$$f = \rho(t^{-1} u^{-1} u_1 v u_1^{-1} t) = \rho(t^{-1} u^{-1} u_1 v u_1^{-1} u t t^{-1} u^{-1} t) =$$
	$$\rho(f_1 v f_1^{-1} t^{-1} u^{-1} t)  = \rho(\beta v \beta^{-1} \alpha u^{-1} \alpha^{-1}),$$
with $\alpha = t^{-1}$ and $\beta = f_1$, proving (\ref{rtl2-I}).
	
The word $u_0:= t^{-1} u_1 f_1^{-1} = \rho(t^{-1} u t)$ is a cyclic permutation of $u$ and $u_0 = \rho(\alpha u^{-1} \alpha^{-1})$, proving (\ref{rtl2-O}).
	
It remains to prove (\ref{rtl2-II}), that is that $f_1 v f_1^{-1}$ is a reduced word. This is indeed the case because $s v s^{-1}$ is a reduced word and $s v s^{-1} = t f_1 v f_1^{-1} t^{-1}$, so $f_1 v f_1^{-1}$ is a subword of it.

\bigskip

\textbf{5.} \hspace{0.1cm}
	\begin{tabular}{|c|c|c|}
		\hline
		\rule{0pt}{2.3ex}
		$\,\, s \,\,$ & $\,\, v \,$ & $\,\, s^{-1} \,\,$ \\  
		\hline
	\end{tabular}
	
	\hspace{0.66cm}
	\begin{tabular}{|c|c|c|c|}
		\hline 
		\rule{0pt}{2.3ex}
		$t$ & $f$ & $t^{-1}$ & $\, u \,$ \\
		\hline
	\end{tabular}	
	
\medskip
	
\noindent There exist words $f_1, f_2, t_1, t_2$ such that $s = t f_1$, $f = f_1 f_2$, $v = f_2 t_2^{-1}$, $t^{-1} = t_2^{-1} t_1^{-1}$, $s^{-1} = t_1^{-1} u$.
	
So $s^{-1} = f_1^{-1} t^{-1} = f_1^{-1} t_2^{-1} t_1^{-1}$, but $s^{-1} = t_1^{-1} u$, then $f_1^{-1} t_2^{-1} t_1^{-1} = t_1^{-1} u$ and by (\ref{remCycPerm}) of Proposition \ref{summar} $f_1^{-1} t_2^{-1}$ is a cyclic permutation of $u$.

The word $u_0 := t_2^{-1} f_1^{-1}$ is a cyclic permutation of $f_1^{-1} t_2^{-1}$ and thus of $u$. Therefore there exists a word $\alpha$ such that $u_0 = \rho(\alpha u \alpha^{-1})$, proving (\ref{rtl2-O}).
	
We have that $f_1 = \rho(u_0^{-1} t_2^{-1})$ and $f_2 = \rho(v t_2)$, so since $\rho(t_2 u_0) = f_1$ then
	$$f = \rho(u_0^{-1} t_2^{-1} v t_2) = \rho(u_0^{-1} t_2^{-1} v t_2 u_0 u_0^{-1}) = \rho(f_1 v f_1^{-1} u_0^{-1}).$$ 
If we set $\beta := f_1$ then $f = \rho(\beta v \beta^{-1} \alpha u^{-1} \alpha^{-1})$, proving (\ref{rtl2-I}). 	
	
It remains to prove (\ref{rtl2-II}), i.e., that is that $f_1 v f_1^{-1}$ is a reduced word. This is indeed the case because $s v s^{-1}$ is a reduced word and $s v s^{-1} = t f_1 v f_1^{-1} t^{-1}$, so $f_1 v f_1^{-1}$ is a subword of it.

\bigskip

\textbf{6.}  \hspace{0.1cm}
	\begin{tabular}{|c|c|c|}
		\hline
		\rule{0pt}{2.3ex}
		$\,\, s \,\,$ & $v$ & $\,\,\,\,\,\,\, s^{-1} \,\,\,\,\,\,\,$ \\  
		\hline
	\end{tabular}
	
	\hspace{0.66cm}
	\begin{tabular}{|c|c|c|c|}
		\hline 
		\rule{0pt}{2.3ex}
		$t$ & $\,\,\,\,\, f \,\,\,\,$ & $t^{-1}$ & $u$ \\
		\hline
	\end{tabular}	
	
\medskip
	
\noindent There exist words $f_1, f_2$ such that $s = t f_1$, $f = f_1 v f_2$, $s^{-1} = f_2 t^{-1} u$, which implies that $s = u^{-1} t f_2^{-1}$. We have that $s^{-1} = f_1^{-1} t^{-1}$, but $s^{-1} = f_2 t^{-1} u$, then 
	\begin{equation} \label{A6.0} f_1^{-1} t^{-1} = f_2 t^{-1} u. \end{equation}
	
This implies that $|f_1| + |t| = |f_2| + |t| + |u|$ and thus that $|f_1| = |f_2| + |u|$. Since $f_1^{-1}$ and $f_2$ are prefixes of the same word and since $|f_2| \leq |f_1|$ then $f_2$ is a prefix of $f_1^{-1}$. Therefore there exists a word $x$ such that $f_1^{-1} = f_2 x$. From (\ref{A6.0}) we have that $f_2 x t^{-1} = f_2 t^{-1} u$ and thus $x t^{-1} = t^{-1} u$, so by (\ref{remCycPerm}) of Proposition \ref{summar} $x$ is a cyclic permutation of $u$.
	
Let us set $u_0 := x$. We have that there exists a word $\alpha$ such that $u_0 = \rho(\alpha u \alpha^{-1})$, proving (\ref{rtl2-O}). We have also that $f_1 = u_0^{-1} f_2^{-1}$, thus $f = u_0^{-1} f_2^{-1} v f_2$. Let us set $\beta := f_2^{-1}$. Then $\beta v \beta^{-1}$ is a reduced word because it is a subword of $f$ which is reduced, proving (\ref{rtl2-II}). Also $f = \rho(\alpha u^{-1} \alpha^{-1} \beta v \beta^{-1})$, proving (\ref{rtl2-I}). 	
	

\bigskip

\textbf{7.}  \hspace{0.1cm}
	\begin{tabular}{|c|c|c|}
		\hline
		\rule{0pt}{2.3ex}
		$s$ & $\,\,\,\,\,\,\,\,\,\,\,\, v \,\,\,\,\,\,\,\,\,\,\,$ & $s^{-1}$ \\  
		\hline
	\end{tabular}
	
	\hspace{0.66cm}
	\begin{tabular}{|c|c|c|c|}
		\hline 
		\rule{0pt}{2.3ex}
		$\,\, t \,\,\,$ & $f$ & $t^{-1}$ & $\,\,\,\,\, u \,\,\,\,$ \\
		\hline
	\end{tabular}	
	
\medskip
	
\noindent There exist words $h, v_1$ such that $t = s h$, $v = h f t^{-1} v_1$, $u = v_1 s^{-1}$. So $t^{-1} = h^{-1} s^{-1}$ and then $v = h f h^{-1} s^{-1} v_1$. The word $u_0 := s^{-1} v_1$ is a cyclic permutation of $u$ and thus we have that $v = h f h^{-1} u_0$. If $h$ is non-empty we have that (\ref{hfh^{-1}u_0}) holds, otherwise $v = f u_0$ and since $v$ is cyclically reduced then $v = f * u_0$ and (\ref{u_0*fSim}) holds.
	
By Lemma \ref{iarAppB} the identity among relations involving $u_0$, $v$, $u_0^{-1}$, $v^{-1}$ is basic and by Remark \ref{iarBasicConj} it is also basic when involving $u$, $v$, $u^{-1}$, $v^{-1}$. 

\bigskip

\textbf{8.}	\hspace{0.1cm}
	\begin{tabular}{|c|c|c|}
		\hline
		\rule{0pt}{2.3ex}
		$s$ & $\,\,\,\, v \,\,\,\,$ & $\,\,\, s^{-1} \,\,$ \\  
		\hline
	\end{tabular}
	
	\hspace{0.66cm}
	\begin{tabular}{|c|c|c|c|}
		\hline 
		\rule{0pt}{2.3ex}
		$\,\, t \,\,$ & $f$ & $t^{-1}$ & $u$ \\
		\hline
	\end{tabular}

\medskip

\noindent There exist words $h, v_1, s_1$ such that $t = s h$, $v = h f v_1$, $t^{-1} = v_1 s_1^{-1}$, $s^{-1} = s_1^{-1} u$. So $t^{-1} = h^{-1} s^{-1} = h^{-1} s_1^{-1} u$, but $t^{-1} = v_1 s_1^{-1}$, therefore 
	\begin{equation} \label{A.1} h^{-1} s_1^{-1} u = v_1 s_1^{-1}. \end{equation}
This implies that $|h| + |s_1| + |u| = |v_1| + |s_1|$ and then $|h| + |u| = |v_1|$. Since $h^{-1}$ and $v_1$ are prefixes of the same word and since $|h| \leq |v_1|$ then $h^{-1}$ is a prefix of $v_1$, therefore there exists a word $x$ such that 
	\begin{equation} \label{A.2} v_1 = h^{-1} x, \end{equation}
which implies that 
	\begin{equation} \label{A.3} v = h f h^{-1} x. \end{equation}
	
From (\ref{A.1}) and (\ref{A.2}) we have that $h^{-1} s_1^{-1} u = h^{-1} x s_1^{-1}$ and thus $s_1^{-1} u = x s_1^{-1}$, so by (\ref{remCycPerm}) of Proposition \ref{summar} we have that $x$ is a cyclic permutation of $u$. If we set $u_0 := x$ then from (\ref{A.3}) we have that $v = h f h^{-1} u_0$, and with the same reasoning as in case 7, we have that either (\ref{hfh^{-1}u_0}) or (\ref{u_0*fSim}) hold.
	
By Lemma \ref{iarAppB} the identity among relations involving $u_0$, $v$, $u_0^{-1}$, $v^{-1}$ is basic and by Remark \ref{iarBasicConj} it is also basic when involving $u$, $v$, $u^{-1}$, $v^{-1}$. 
	

\bigskip

\textbf{9.}  \hspace{0.1cm}
	\begin{tabular}{|c|c|c|}
		\hline
		\rule{0pt}{2.3ex}
		$s$ & $v $ & $\,\,\,\,\,\,\, s^{-1} \,\,\,\,\,\,$ \\  
		\hline
	\end{tabular}
	
	\hspace{0.66cm}
	\begin{tabular}{|c|c|c|c|}
		\hline 
		\rule{0pt}{2.3ex}
		$\,\, t \,\,$ & $f$ & $t^{-1}$ & $u$ \\
		\hline
	\end{tabular}	
	
\medskip	
	
\noindent This case is impossible because $|t| \geq |s|$ and $|s| \geq |t| + |u|$, which implies that $|t| \geq |t| + |u|$. Thus $|u| = 0$ and therefore $u = 1$ and this is excluded by hypothesis.

\bigskip

\textbf{10.} \hspace{0.1cm}
	\begin{tabular}{|c|c|c|}
		\hline
		\rule{0pt}{2.3ex}
		$s$ & $v $ & $\,\,\,\,\,\,\,\,\,\,\, s^{-1} \,\,\,\,\,\,\,\,\,\,\,$ \\  
		\hline
	\end{tabular}
	
	\hspace{0.88cm}
	\begin{tabular}{|c|c|c|c|}
		\hline 
		\rule{0pt}{2.3ex}
		$\,\,\,\,\,\,\, t \,\,\,\,\,\,$ & $f$ & $t^{-1}$ & $u$ \\
		\hline
	\end{tabular}	
	
\medskip	
	
\noindent This case is impossible because $|t| \geq |s| + |w|$ and  $|s| \geq |f| + |t| + |u|$, implying that  $|t| \geq |f| + |t| + |u| + |w|$, therefore $|f| + |u| + |w| = 0$, that is $f$, $u$ and $v$ should be empty words, which contradicts the hypothesis that $u \neq 1$.	
\end{proof}

\begin{proposition} \label{tecLemma2a} Let $f, t, u$ be words such that $u t f t^{-1}$ is a reduced word. Let us set $v := \hat{\rho}(u t f t^{-1})$. Then there exists a cyclic permutation $u_0$ of $u$ such that either
	\begin{equation} \label{u_0*fSim.a} v \sim u_0*f \end{equation}	
(in particular $v = u_0*f$ or $v = f*u_0$) or there exists a non-empty word $h$ such that 
	\begin{equation} \label{u_0hfh^{-1}} v = u_0 h f h^{-1}.\end{equation} 	
	
Finally the identities among relations involving $u, v, u^{-1}, v^{-1}$ that by Remark \ref{idFrom} follow from (\ref{u_0*fSim.a}) and (\ref{u_0hfh^{-1}}) are basic.
\end{proposition}

\begin{proof} By (\ref{revCycRedForm}) and (\ref{reverse2}) of Proposition \ref{summar} we have that $\underline{v } = \hat{\rho}(\underline{t}^{-1} \, \underline{f} \, \underline{t} \, \underline{u})$. 
	
By Proposition \ref{tecLemma2} there exists a cyclic permutation $\underline{u_0}$ of $\underline{u}$ such that either 
	\begin{equation} \label{yalApp} \underline{v } \sim \underline{u_0}*\underline{f} \end{equation}
(in particular $\underline{v } = \underline{u_0}*\underline{f}$ or $\underline{v } = \underline{f}*\underline{u_0}$) or there exists a non-empty word $h$ such that 
	\begin{equation} \label{yalApp2} \underline{v } = \underline{h}^{-1} \, \underline{f} \, \underline{h} \, \underline{u_0}. \end{equation} 	
Also by Proposition \ref{tecLemma2} the identities among relations involving $\underline{u}, \underline{v }, \underline{u}^{-1}, \underline{v }^{-1}$ that by Remark \ref{idFrom} follow from (\ref{yalApp}) and from (\ref{yalApp2}) are basic.
	
Let $\underline{v } = \underline{u_0}*\underline{f}$; we have that 
	$$\underline{u_0}*\underline{f} = \hat{\rho}(\underline{u_0} \, \underline{f}) = \hat{\rho}(\underline{f u_0}) = \underline{\hat{\rho}(f u_0)} = \underline{f *u_0},$$
where the second and third equalities follow respectively from (\ref{reverse2}) and (\ref{revCycRedForm}) of Proposition \ref{summar}. This implies that $v = u_0*f$.

In the same way we prove that if $\underline{v } = \underline{f}*\underline{u_0}$ then $v = f*u_0$. Thus in these two cases we have that $v \sim u_0*f$ by Proposition \ref{puzo}, thus (\ref{u_0*fSim.a}) holds.

Let there exist a non-empty word $h$ such that $\underline{v } = \underline{h}^{-1} \, \underline{f} \, \underline{h} \, \underline{u_0}$; then $v = u_0 h f h^{-1}$ by (\ref{reverse2}) of Proposition \ref{summar}, thus (\ref{u_0hfh^{-1}}) holds.

Finally the identities among relations that follow from (\ref{u_0*fSim.a}) or (\ref{u_0hfh^{-1}}) are basic by Remark \ref{idRev}.
\end{proof}

\begin{remark} \label{remAftTecLem} \rm In (\ref{hfh^{-1}u_0}) of Proposition \ref{tecLemma2} and in (\ref{u_0hfh^{-1}}) of Proposition \ref{tecLemma2a} the word $u_0$ is reduced because it is a subword of $v$, which is reduced. Since $u_0$ is a cyclic permutation of $u$, then either $u_0 = u$ or $u$ is cyclically reduced.
\end{remark}

\begin{lemma} \label{tecLemma2b} Let $f, t, u_1, u_2$ be words such that $u_1 \neq 1$ and $u_1 t f t^{-1} u_2$ is a reduced word. Let us set $u := u_1 u_2$ and $v := \hat{\rho}(u_2 u_1 t f t^{-1})$. Then there exists a word $u_0$ which is the reduced form of a cyclic permutation of $u$ such that either
	\begin{equation} \label{u_0*fSim.b} v \sim u_0*f \end{equation}	
or there exists a non-empty word $h$ such that 
	\begin{equation} \label{u_0hfh^{-1}.b} v \sim u_0 h f h^{-1}.\end{equation} 	
	
Finally the identities among relations involving $u, v, u^{-1}, v^{-1}$ that by Remark \ref{idFrom} follow from (\ref{u_0*fSim.b}) and (\ref{u_0hfh^{-1}.b}) are basic.
\end{lemma}

\begin{proof} We prove the result by induction on the length of $|u_2|$: if $|u_2| = 0$ the claim follows from Proposition \ref{tecLemma2a}, so we can assume that the claim is true for every $u_2'$ such that $|u_2'| < |u_2|$.

Let $v_0 := \rho(u_2 u_1 t f t^{-1})$, that is $v := \hat{\rho}(v_0)$. Then there exist words $u_3, u_4, a, t_1, t_2, f_1, f_2$ such that one of the following four cases holds: \begin{enumerate}
	\item $u_2 = u_3 a$, $u_1 = a^{-1} u_4$, $v_0 = u_3 u_4 t f t^{-1}$;

	\item $t = t_1 t_2$, $u_2 = u_3 t_1^{-1} u_1^{-1}$, $v_0 = u_3 t_2 f t^{-1}$;

	\item $f = f_1 f_2$, $u_2 = u_3 f_1^{-1} t^{-1} u_1^{-1}$, $v_0 = u_3 f_2 t^{-1}$;

	\item $t = t_1 t_2$, $u_2 = u_3 t_2 f^{-1} t^{-1} u_1^{-1}$, $v_0 = u_3 t_1^{-1}$.
\end{enumerate}

Let us prove the claim in the four cases.

\smallskip

1. Let us set $u' := u_3 u_4$, that is $v_0 = u' t f t^{-1}$. By applying Proposition \ref{tecLemma2a} to the reduced word $u' t f t^{-1}$ we have that there exists a cyclic permutation $u_0$ of $u'$ such that either $v \sim u_0*f$ or there exists a non-empty word $h$ such that $v = u_0 h f h^{-1}$; in the latter case, by Remark \ref{remAftTecLem} either $u_0 = u'$ or $u'$ is cyclically reduced. Finally the identities among relations involving $u', v, (u')^{-1}, v^{-1}$ that by Remark \ref{idFrom} follow from them are basic. By Remark \ref{iarBasicConj} these identities, considered as identities involving $u, v, u^{-1}, v^{-1}$ are basic too.

We have that $u = a^{-1} u_4 u_3 a$. Since $u_0$ is a cyclic permutation of $u'$, two cases are possible: either there exist words $u_3', u_3''$ such that $u_3 = u_3' u_3''$ and $u_0 = u_3'' u_4 u_3'$ or there exist words $u_4', u_4''$ such that $u_4 = u_4' u_4''$ and $u_0 = u_4'' u_3 u_4'$. In particular in the first case $\rho(u_0) = \rho(u_3'' a a^{-1} u_4 u_3')$ and in the second case $\rho(u_0) = \rho(u_4'' u_3 a a^{-1} u_4')$, so in both cases $\rho(u_0)$ is the reduced form of a cyclic permutation of $u$.

Let us suppose that $v \sim u_0*f$; then since $u_0*f = \rho(u_0) * f$, by replacing $u_0$ with $\rho(u_0)$ the claim is proved. Now let us suppose that $v = u_0 h f h^{-1}$. If $u_0 = u'$ then $u_0$ is the reduced form of $u_3 a a^{-1} u_4$, which is a cyclic permutation of $u$, proving the claim. If $u'$ is cyclically reduced then $u_0$ is reduced, thus $\rho(u_0) = u_0$ and by what seen above it is the reduced form of a cyclic permutation of $u$.

\smallskip

2. We have that $u = u_1 u_3 t_1^{-1} u_1^{-1}$ and $v_0 = u_3 t_2 f t_2^{-1} t_1^{-1}$. Since $u_2 = u_3 t_1^{-1} u_1^{-1}$ and since we have assumed that $u_1 \neq 1$, then $|t_1^{-1}| < |u_2|$. Let us set $v' := \hat{\rho}(t_1^{-1} u_3 t_2 f t_2^{-1})$ and $u' := t_1^{-1} u_3$; by induction hypothesis there exists a word $u'_0$ which is the reduced form of a cyclic permutation of $u'$ such that either $v' \sim u'_0*f$ or there exists a non-empty word $h$ such that $v' \sim u'_0 h f h^{-1}$. Moreover the identities among relations involving $u', v', (u')^{-1}, (v')^{-1}$ that by Remark \ref{idFrom} follow from them are basic.

By Corollary \ref{cycRedPerm} we have that $v'$ is a cyclic permutation of $v$, so $v \sim u'_0*f$ or there exists a non-empty word $h$ such that $v \sim u'_0 h f h^{-1}$. By Remark \ref{basicToo} the identities among relations involving $u', v, (u')^{-1}, v^{-1}$ that by Remark \ref{idFrom} follow from them are basic. By Remark \ref{iarBasicConj} these identities, considered as identities involving $u, v, u^{-1}, v^{-1}$ are basic too.

It remains to prove that $u_0'$ is the reduced form of a cyclic permutation of $u$. Since $u'_0$ is the reduced form of a cyclic permutation of $u'$ and since $u' = t_1^{-1} u_3$ then there exist words $x, y$ such that either $t_1^{-1} = x y$ and $u'_0 = y u_3 x$ or $u_3 = x y$ and $u'_0 = y t_1^{-1} x$. In the first case $u = u_1 u_3 x y u_1^{-1}$, in the second case $u = u_1 x y t_1^{-1} u_1^{-1}$. In both cases $u_0'$ is the reduced form of a cyclic permutation of $u$ and this completes the proof.

\smallskip

3. We have that $u = u_1 u_3 f_1^{-1} t^{-1} u_1^{-1}$. Let us set $\alpha := t^{-1} u_1^{-1}$, $u_0 := \rho(\alpha u \alpha^{-1})$ and $\beta := t^{-1}$. Then $u_0 = \rho(t^{-1} u_3 f_1^{-1})$ and 
	$$\rho(\alpha u^{-1} \alpha^{-1} \beta v_0 \beta^{-1}) = \rho(f_1 u_3^{-1} t t^{-1} u_3 f_2 t^{-1} t) = f_1 f_2 = f.$$
Since $u_0$ is the reduced form of a cyclic permutation of $u$, the claim follows from Lemma \ref{baseLemma}.


\smallskip

4. We have that $u = u_1 u_3 t_2 f^{-1} t^{-1} u_1^{-1}$. Let us set $\alpha := t^{-1} u_1^{-1}$, $u_0 := \rho(\alpha u \alpha^{-1})$ and $\beta := t^{-1}$. Then $u_0 = \rho(t^{-1} u_3 t_2 f^{-1})$ and 
	$$\rho(\alpha u^{-1} \alpha^{-1} \beta v_0 \beta^{-1}) = \rho(f t_2^{-1} u_3^{-1} t t^{-1} u_3 t_1^{-1} t_1 t_2) = f.$$
Since $u_0$ is the reduced form of a cyclic permutation of $u$, the claim follows from Lemma \ref{baseLemma}. \end{proof}

\begin{proposition} \label{tecLemma2c} Let $f, t, u_1, u_2$ be words such that $u_1 t f t^{-1} u_2$ is a reduced word. Let us set $u := u_1 u_2$ and $v := \hat{\rho}(u_1 t f t^{-1} u_2)$. Then there exists a word $u_0$ which is the reduced form of a cyclic permutation of $u$ such that either
	\begin{equation} \label{u_0*fSim.c} v \sim u_0*f \end{equation}	
or there exists a non-empty word $h$ such that 
	\begin{equation} \label{u_0hfh^{-1}.c} v \sim u_0 h f h^{-1}.\end{equation} 	
	
Finally the identities among relations involving $u, v, u^{-1}, v^{-1}$ that by Remark \ref{idFrom} follow from (\ref{u_0*fSim.c}) and (\ref{u_0hfh^{-1}.c}) are basic.
\end{proposition}

\begin{proof} If $u_1 = 1$ the claim follows from Proposition \ref{tecLemma2}. Let $u_1 \neq 1$; by Lemma \ref{tecLemma2b} we have that if we set $v' := \hat{\rho}(u_2 u_1 t f t^{-1})$ then there exists a word $u_0$ which is the reduced form of a cyclic permutation of $u$ such that either $v' \sim u_0*f$ or there exists a non-empty word $h$ such that $v' \sim u_0 h f h^{-1}$. Moreover the identities among relations involving $u, v, u^{-1}, v^{-1}$ that by Remark \ref{idFrom} follow from these equivalences are basic.

The first part of the claim follows from this lemma because $v \sim v'$; the second part follows from Remark \ref{basicToo}.
\end{proof}

\newpage

\textit{Address:}

Carmelo Vaccaro

Laboratoire de Mathématiques d'Orsay

Université Paris-Saclay 

Bâtiment 307, rue Michel Magat

91400 Orsay

\medskip

\textit{e-mail:} \textsf{carmelo.vaccaro@universite-paris-saclay.fr}

 \end{document}